\documentclass[]{dependencies/interact}

\usepackage{dependencies/dependencies}
\graphicspath{{images/}}

\usepackage{epstopdf}% To incorporate .eps illustrations using PDFLaTeX, etc.
\usepackage[caption=false]{subfig}% Support for small, `sub' figures and tables

\usepackage[numbers,sort&compress]{natbib}% Citation support using natbib.sty
\bibpunct[, ]{[}{]}{,}{n}{,}{,}% Citation support using natbib.sty
\renewcommand\bibfont{\fontsize{10}{12}\selectfont}% Bibliography support using natbib.sty

\theoremstyle{plain}% Theorem-like structures provided by amsthm.sty
\newtheorem{theorem}{Theorem}[section]
\newtheorem{lemma}[theorem]{Lemma}
\newtheorem{corollary}[theorem]{Corollary}

\theoremstyle{definition}

\theoremstyle{remark}
\newtheorem{remark}{Remark}

\begin{document}

% \articletype{ARTICLE TEMPLATE}% Specify the article type or omit as appropriate

\title{Generating Linear, Semidefinite, and Second-order Cone Optimization Problems for Numerical Experiments}

\author{ 
\name{Mohammadhossein Mohammadisiahroudi\textsuperscript{a}\thanks{Corresponding author: M.~M., mom219@Lehigh.edu}, 
Ramin Fakhimi\textsuperscript{a}, 
Brandon Augustino\textsuperscript{a}, and 
Tam\'as  Terlaky\textsuperscript{a}}
\affil{\textsuperscript{a}Industrial and System Engineering Department, Lehigh University, Bethlehem, PA, USA}
}

\maketitle

\begin{abstract}
The numerical performance of algorithms can be studied using test sets or procedures that generate such problems. 
This paper proposes various methods for generating linear, semidefinite, and second-order cone optimization problems.
Specifically, we are interested in problem instances requiring a known optimal solution, a known optimal partition, a specific interior solution, or all these together. 
In the proposed problem generators, different characteristics of optimization problems, including dimension, size, condition number, degeneracy, optimal partition, and sparsity, can be chosen to facilitate comprehensive computational experiments. 
We also develop procedures to generate instances with a maximally complementary optimal solution with predetermined optimal partition to generate challenging semidefinite and second-order cone optimization problems.
Generated instances enable us to evaluate efficient interior-point methods for conic optimization problems. 
\end{abstract}


\section{Introduction}

Optimization is just one of many fields in which the empirical analysis of algorithms is heavily reliant on the quality of the provided test instances. 
Scholars assess the strengths and weaknesses of algorithms based on these test problems, which must be unbiased, representative, and diverse in their measurable features or characteristics. 
However, many benchmark test problems do not possess these desired qualities, as they are often based on a limited set of real-world problems or have been reused from earlier studies that by now may be obsolete \cite{bowly2019generation}.

An alternative approach is using random test problem generators for experimentation in optimization. 
While their design must be carefully considered, one advantage of simple random generation approaches is their ability to produce problems that possess predictable characteristics. 
As a result, scientists have advocated for using highly parameterized generators to produce appropriately controlled data for experimentation \citep{hooker1994needed}. 
As one of the first attempts in this area, randomly generated feasible polyhedra properties were investigated by Todd~\cite{todd1991probabilistic}. 
Pilcher and Rardin~\cite{pilcher1992partial} proposed a generator for pure integer optimization problems with a known partial polytope by introducing random cuts. 
Yet, this methodology is restricted to traveling salesman problems and does not explicitly consider the solution of relaxation or structural features. 
Lacking the ability to vary features of interest, the scope of these generators for experimentation is limited to specific problem domains.

At times, it can be challenging to develop instance generators in a way that allows properties of interest to be suitably varied. 
While specific characteristics, such as the density of a graph, can usually be directly controlled through the generation process, other attributes can be harder to predefine or control explicitly. 
Many measurable features of the same problem instance can be highly correlated, either due to interacting bounds or simply as a consequence of the random generation process. 
Instances with less-like feature combinations can be attained through an iterative local search, which successively modifies an instance until it possesses the desired properties. 
While these instance-space search techniques are more computationally intensive than parameterized generators, they provide a reliable method for producing instances with specific target characteristics \citep{bowly2019generation}.

The most prevalent search techniques for this application are evolutionary algorithms. 
Chakraborty and Choudhury~\cite{chakraborty2000statistical} and Cotta and Moscato~\cite{cotta2003mixed} applied this approach to perform statistical average- and worst-case analysis of algorithm performance. 
More recently, exploration in this direction has focused on improving the spectrum of instance hardness and diversity of measured features \citep{smith2015generating}.
The success of these techniques in combinatorial optimization opens up questions on the use of similar approaches for linear optimization (LO) and mixed-integer optimization, adopting a more comprehensive range of search algorithms for obtaining difficult-to-design instances, and considering how to best construct the search space for efficient performance.

To develop instance generation techniques for LO test problems with controllable properties, Bowly et al.~\cite{bowly2019generation} presented a comparison of a naive random generator with a highly parameterized generator, showing which feature values can be effectively controlled by each method. 
They also investigated iterative search approaches to find instances that are difficult to design or rarely produced by the generator. 
These approaches allow practitioners to explore areas of interest in the space of linear optimization problems (LOPs), where challenging instances have previously been found. 
This would be impossible using static test sets or n\"aive random generation methods, which provide limited feature control. 
Further, large-scale linear optimization problems are prevalent in economics, industry, logistics, statistics, quantum physics, and other fields. 
As is the case with any real-world application, the aim is to obtain high-quality solutions efficiently, a task for which high-performance computing systems and parallel algorithms are required. 
Thus, the development of new parallel algorithms for generating LOPs and the revision of current algorithms are considered by Sokolinsky and Sokolinskaya~\cite{sokolinsky2021fragenlp}.

Developing new algorithms for solving large-scale LOPs necessitates testing them on benchmark and random problems.
At times, it is sensible to construct linear and integer optimization instance generators specified for special purposes. 
The NETGEN generator \citep{klingman1974netgen} and its successor MNETGEN produce parameterized multicommodity flow, transport, and assignment problems.
The parameters used are thus appropriate to the underlying network, not the feasible set. 
One of the well-known benchmark repositories of LOPs is Netlib-LP \citep{gay1985electronic}. Yet, when debugging LO solvers, generating random LOPs with specific characteristics (such as, e.g., the sparsity, condition number of the coefficient matrix, or a known optimal partition) is often necessary.

Charnes et al.~\cite{charnes1974generation} suggested one of the first methods for generating random LOPs with known solutions. 
This method allows one to generate test problems of arbitrary size with a wide range of numerical characteristics. 
The main idea of the method is as follows; take as a basis a LOP with a known solution, and then randomly modify it so that the solution does not change. 
The key drawback  of this approach is that fixing the optimal solution in advance significantly restricts the random nature of the resulting LOP.

Arthur and Frendewey~\cite{arthur1993gengub} described the GENGUB generator, which constructs random LOPs with a known solution and given characteristics, such as the problem size, the density of the coefficient matrix, the number of binding inequalities, or the degeneracy status. 
A distinctive feature of GENGUB is the ability to introduce generalized upper bound constraints, defined to be a (sub)set of constraints in which each variable appears at most once (i.e., has at most one nonzero coefficient). 
This method has similar drawbacks to the generator found in \cite{charnes1974generation}: by fixing the optimal solution ex ante, the random nature of the resulting LOP is significantly restricted.

Castillo et al.~\cite{castillo2001automatic} suggest a method for generating random LOPs with a preselected solution type: bounded or unbounded, unique or multiple.
Each structure is generated using random vectors with integer components, whose range can be treated as given.
Next, an objective function that satisfies the required conditions, i.e., leads to a solution of the desired type, is obtained. 
This LO problem generator is mainly used for educational purposes rather than testing new LO algorithms.
Okolinsky and Sokolinskaya~\cite{sokolinsky2021fragenlp} proposed the random LOP generator FRaGenLP (Feasible Random Generator of LP), which is implemented as a parallel program for cluster computing systems. 
Calamai et al.~\cite{calamai1993new} described a new technique for generating convex, strictly concave, and indefinite (bilinear or not) quadratic optimization problems.

In the semidefinite optimization literature, scholars were interested in complex problems. 
They pursued various directions for characterizing what constitutes hardness in SDO problems, e.g., not having a strictly complementary solution \citep{pataki2017bad}, or a solution with a nonzero duality gap \citep{sremac2021error}.
Wei and Wolkowicz~\cite{wei2010generating} proposed a procedure to generate SDO problems without a strictly complementary solution. 
We build on these ideas to develop highly parameterized generators. 

\subsection{Contributions}
This paper reviews and proposes several procedures to generate random LOPs, semidefinite optimization problems (SDOPs), and second-order cone optimization problems (SOCOPs) with a specified optimal solution, interior solution, and both of them. We also develop SDOP and SOCOP generators with specific maximally complementary solutions to predetermine the optimal partition. 

Generating SDOPs and SOCOPs with a specific interior solution ensures that Strong Duality holds for the generated problems, and the set of optimal solutions will be bounded. 
Access to predefined interior solutions will enable researchers to analyze the performance of optimization algorithms, such as feasible Interior Point Methods (IPMs), with respect to various initial interior solutions. 

Generating problems with known optimal solutions ensures that the generated problem has a bounded optimum and helps to analyze the algorithm concerning the characteristics of the optimal solution. 
These procedures will serve to further scholars' ability to examine their algorithms by altering different features of input data such as dimension, sparsity, condition number, solution size (which plays an essential role in the performance of Infeasible IPMs), and many others, besides predefined properties of the optimal solution. 
Another possible application of the proposed procedures is the average-case complexity analysis of algorithms. 

The rest of the paper is organized as follows. 
In Section~\ref{sec: LO}, we give a brief review of LO theory before considering several LOP generators that can generate instances with specific optimal solutions, specific interior solutions, or both. 
We then develop similar generators for SDO and SOCO in Sections~\ref{sec: SDO} and~\ref{sec: SOCO}, respectively. 
A discussion on the implementation of the proposed instance generators is provided in Section~\ref{sec: imp}, and Section~\ref{sec: con} concludes the paper.

\section{Linear Optimization}\label{sec: LO}
In this section, we provide a gentle review of Linear Optimization theory before presenting three different algorithms for randomly generating Linear Optimization test problems.

\subsection{Linear Optimization Problems}\label{sec: LOP}
In LOPs, we seek to minimize the inner product of two $n$-dimensional vectors
$$   c^{\top} x = \sum_{i= 1}^n c_i \cdot x_i,$$
for a constant vector $c \in \R{n}$ and variable vector $x \in \R{n}$. 
In this minimization, variable $x$ must satisfy linear constraints of the form 
$$ Ax = b,$$
for a given matrix $A \in \R{m \times n}$ and vector $b \in \R{m}$. Moreover, we require that $x$ be elementwise nonnegative, which we denote by $x \geq 0$. 

We are therefore interested in randomly generating LOPs of the form
\begin{equation}\label{e:LO-P}
    z^P_{LO} = \min_x \left\{c^{\top} x : Ax = b, x \geq 0\right\}, \tag{LOP-P}
\end{equation}
and refer to \eqref{e:LO-P} as the \textit{primal problem}. Given the primal problem \eqref{e:LO-P}, we are also interested in a second problem known as the \textit{dual problem} of \eqref{e:LO-P}, which we write in standard form as follows,
\begin{equation}\label{e:LO-D}
    z^D_{LO} = \max_{(y,s)} \left\{b^{\top} y : A^{\top} y + s = c,  s \geq 0, y \in \R{m} \right\}, \tag{LOP-D}
\end{equation}
where  $s = c - A^{\top} y$ is the dual slack variable. 

We say that $x$ and $(y,s)$ are \textit{feasible solutions} whenever they satisfy the constraints of the primal and dual problems, respectively. The set of primal-dual feasible solutions is thus defined as
\begin{equation*}
\Pcal \Dcal_{LO}=\left\{(x,y,s)\in\mathbb{R}^{n} \times \mathbb{R}^{m} \times \mathbb{R}^{n}: Ax=b, A^{\top} y+s=c,  (x,s)\geq0\right\}.
\end{equation*}
Similarly, the set of all \textit{feasible interior solutions} is given by
\begin{equation*}
\Pcal \Dcal_{LO}^0=\left\{(x,y,s)\in\Pcal \Dcal_{LO} : (x,s)>0\right\}.
\end{equation*}
A crucial property of linear optimization is \textit{weak duality}; any $(y,s)$ that is feasible for \eqref{e:LO-D}, provides a lower bound $b^{\top} y$ on the value of $c^{\top} x$ for any $x$ feasible for \eqref{e:LO-P}, i.e.: 
$$b^{\top} y \leq c^{\top} x,$$
for any $(x, y, s) \in \Pcal \Dcal_{LO}$. Conversely, any $x$ that is feasible for \eqref{e:LO-P} provides an upper bound $c^{\top} x$ on $b^{\top} y$ for any $y$ that is feasible for \eqref{e:LO-D}, and we refer to the nonnegative quantity $c^{\top} x - b^{\top} y=x^{\top}s $ as the \textit{duality gap}. 

Whenever $(x, y, s) \in \Pcal \Dcal$ with $c^{\top} x = b^{\top} y$, or equivalently $x^{\top}s=0$, then $x$ is optimal for \eqref{e:LO-P} and $(y,s)$ is optimal for \eqref{e:LO-D}. In this case, \textit{strong duality} holds for LOPs, i.e., if both the primal and dual problems have feasible solutions, then both have optimal solution with equal objective value. Under strong duality, all optimal solutions, if there exist any, belong to the
set $\Pcal \Dcal_{LO}^*$, defined as
\begin{equation*}
\Pcal \Dcal_{LO}^*=\left\{(x,y,s)\in\Pcal \Dcal_{LO}  :  x^{\top} s=0\right\}.
\end{equation*}
Let $[n]$ denote the set $\left\{1,2,\dots,n\right\}$. Following Roos et al.~\cite{roos1997theory}, LOPs admit an optimal partition $\Ncal\cup \Bcal=[n]$, and $\Bcal\cap \Ncal = \emptyset$, where 
\begin{align*}
    \Bcal&=\{i: \exists (x^*,y^*,s^*)\in \Pcal \Dcal_{LO}^* \text{ with } x^*_i>0\},\\
    \Ncal&=\{i:\exists (x^*,y^*,s^*)\in \Pcal \Dcal_{LO}^* \text{ with } s^*_i>0\}.
\end{align*}
If $(x^*,y^*,s^*)\in \Pcal \Dcal_{LO}^*$ with $x^*_i>0$ for all $i\in \Bcal$, and $s^*_i>0$ for all $i\in \Ncal$, then we have $x^*+s^*>0$ and the optimal solution pair $(x^*,y^*,s^*)$ is called strictly complementary. In this section, we use $(\Bcal,\Ncal)$ to denote the optimal partition, and $(B,N)$ the index set partition in the algorithms. After presenting each algorithm, we clarify when the predefined partition $(B,N)$ is equal to the optimal partition $(\Bcal,\Ncal)$.

\subsection{Instance Generators for LOPs}\label{sec: LOPGen}
In the rest of this section, we review three main generators which produce LO instances given either a predefined (or randomly chosen) interior solution, a predefined (or randomly chosen) optimal solution (maybe strictly complementary or not), or both. Each LOP generator allows the user to control the characteristics of parameters $(A,b,c)$, including but not limited to their condition number, sparsity, and norm. Further, users can alter the optimal solution's features to examine their algorithm's performance.

In the following algorithms, the term ``generate" should be interpreted freely. It may refer to generating the respective data randomly, or the connotation could be that the data is constructed with some specific purpose, e.g., to obtain matrices with some specific structure such as sparsity or conditioning.

\subsubsection{LOPs with a Predefined Interior Solution}\label{sec: LOPINT}

To study the performance of IPMs applied to LOPs, it is often helpful to have instances with specific interior solutions, and a common approach to generating LOPs with a desired interior solution is presented as Algorithm~\ref{alg: LO-IP}.

\begin{algorithm}[H] 
\caption{Generating a LOP with a specific interior solution} \label{alg: LO-IP}
\begin{algorithmic}[1]
\State Choose dimensions $m<n$
\State Choose or generate $(x^0,s^0)$  such that $x_i^0>0$ and $s_i^0>0$ for all $i\in [n]$
\State Generate $A \in \mathbb{R}^{m\times n}$ 
\State Generate $y^0 \in \mathbb{R}^{m}$ 
\State Calculate $b=Ax^0$ and $c=A^{\top} y^0+s^0$
\State \textbf{Return }{LOP $(A, b, c)$ with interior solution
$(x^0,y^0,s^0)$}
\end{algorithmic}
\end{algorithm}

\begin{remark}
Suppose we want the interior solution $(x^0,s^0)$ to have a duality gap of $x^{0^{\top}} s^0 = n \mu$ for some scalar $\mu>0$. Then, in Step 1 of Algorithm~\ref{alg: LO-IP}, we generate $x_i^0>0$ and calculate $s_i^0=\frac{\mu}{x_i^0}$ for $i\in [n]$.
\end{remark}

The above remark makes an observation relevant to IPMs, as in the context of IPMs, the constant $\mu$, referred to as the central path parameter, plays a crucial role. IPMs begin with some initial interior solution $(x^0,s^0) \in \Pcal \Dcal^0$ with 
$$\frac{x^{0^{\top}} s^0}{n} =\mu^0 > 0,$$
and subsequently, reduce $\mu$ in each iteration as the algorithm progresses toward a solution to the LOP with desired complementarity gap. In line with our discussion on LO duality, it is easy to see that when $\mu \to 0$, we approach an optimal solution to the primal-dual pair \eqref{e:LO-P}-\eqref{e:LO-D}. 

\begin{remark}
Algorithm~\ref{alg: LO-IP} facilitates the generation of a coefficient matrix $A$ with any desired properties, e.g., sparsity, structure, or being ill-conditioned.
\end{remark}

\begin{remark}
Several conditions are needed to generate a full row rank coefficient matrix $A$ with probability one randomly  \citep[see e.g., ][]{coja2021full}.
\end{remark}

\subsubsection{LOPs with a Predefined Optimal Solution}\label{sec: LOPOPT}

A prevailing approach for generating LOPs with a known optimal solution is described in Algorithm~\ref{alg: LO-OP}.

\begin{algorithm}[H] 
\caption{Generating a LOP with a specific optimal solution} \label{alg: LO-OP}
\begin{algorithmic}[1]
\State Choose dimensions $m<n$
\State Partition the index set $[n]$ to $B$ and $N$ with $B\cap N=\emptyset$ and $B\cup N=[n]$
\State Generate $x^*$ such that $x_i^*>0$  for $i\in B$ and $x_i^*=0$ for $i\in N$
\State Generate $s^*$ such that $s_i^*>0$  for $i\in N$ and $s_i^*=0$ for $i\in B$
\State Generate $A \in \mathbb{R}^{m\times n}$ 
\State Generate $y^* \in \mathbb{R}^{m}$  
\State Calculate $b=Ax^*$ and $c=A^{\top} y^*+s^*$
\State \textbf{Return }{LOP $(A, b, c)$ with optimal solution
$(x^*,y^*,s^*)$}
\end{algorithmic}
\end{algorithm}

\begin{remark}
Since the generated optimal solution $(x^*,y^*,s^*)$ by Algorithm~\ref{alg: LO-OP} is strictly complementary, the optimal partition $(\Bcal, \Ncal)$ is equal to $(B,N)$.
\end{remark}

\begin{remark}
Partition $(B, N)$ may be generated randomly or to satisfy some desired properties, such as primal or dual degeneracy, or both, or having a unique optimal basis solution. 
\end{remark}

\begin{remark}
Let $A= [A_B \ A_N]$. If $|B|=m$ and $A_B$ is nonsingular, then $x^*$ and $s^*$ yield the unique optimal basis solution. 
\end{remark}
\begin{remark}
If we modify Algorithm~\ref{alg: LO-OP} by generating $x^*$ such that $x_i^*\geq0$  for $i\in B$ and $x_i^*=0$ for $i\in N$, and $s^*$  such that $s_i^*\geq0$  for $i\in N$ and $s_i^*=0$ for $i\in B$, then $B$ and $N$ do not necessarily give the optimal partition. While $x^*$ and $s^*$ are complementary solutions, they are not necessarily strictly complementary. 
\end{remark}

\subsubsection{LOPs with Predefined Optimal and Interior Solutions}\label{sec: LOPINTOPT}
Charnes et al.~\cite{charnes1974generation} discuss procedures to generate problems with a specific optimal \textit{or} interior solution. Here, we develop a novel procedure to generate a LOP with a specific optimal solution $(x^*,y^*,s^*)$ \textit{and} a specific interior solution $(x^0,y^0,s^0)$,  as presented in Algorithm~\ref{alg: LO-OP-IP}. The general idea is first to use Algorithm~\ref{alg: LO-OP} to generate a problem with optimal solution $(x^*,y^*,s^*)$ before extending the problem by adding a variable and a constraint to make the interior point $(x^0,y^0,s^0)$ feasible for the new problem. Using this scheme, we can produce LOPs for any general predefined optimal and interior solutions, where the only additional condition is
\begin{equation}\label{con1}
    (x^0-x^*)^{\top} (s^0-s^*)=0.
\end{equation}

The condition stipulated by equation~\eqref{con1} is a natural property; for any feasible solution pairs, it follows that $(x^*-x^0)\in \operatorname{Lin}^{\perp}(A)$ and $(s^*-s^0)\in \operatorname{Lin}(A)$, where $\operatorname{Lin}(A)$ denotes the lineality space of $A$. In other words, the difference of the predefined solutions $x^0-x^*$ and $s^0-s^*$ \textit{must} be orthogonal, and steps \ref{alg3st8} and \ref{alg3st9} of Algorithm~\ref{alg: LO-OP-IP} ensure this property holds.

\begin{algorithm}[H] 
\caption{Generating LOP with specific optimal and interior solutions} \label{alg: LO-OP-IP}
\begin{algorithmic}[1]
\State Choose $m<n$ (The generated LOP has $m+1$ constraints and $n+1$ variables.)
\State Generate LOP $(\Ahat, \bhat, \chat)$ with optimal solution
$(\xhat,\yhat,\shat)$ and partition $(B,N)$ using Algorithm \ref{alg: LO-OP}
\State Generate $x^0,s^0\in \mathbb{R}^n$ such $x^0,s^0>0$ 
\State Let $\delta=(x_B^0-\xhat_B)^{\top} s_B^0+(s_N^0-\shat_N)^{\top} x_N^0$, generate $x^0_{n+1}>0$ and $s^0_{n+1}>(\frac{-\delta}{x^0_{n+1}})^+$ \label{alg3st8} 

\State Calculate $\hat{s}_{n+1}=\frac{\delta}{x^0_{n+1}}+s^0_{n+1}$ and let $\hat{x}_{n+1}=0$ \label{alg3st9}

\State Build $x^*=\begin{pmatrix}\hat{x}_B\\0\\0\end{pmatrix}$ and $x^0=\begin{pmatrix}x^0_B\\x^0_N\\x^0_{n+1}\end{pmatrix}$ 

\State Build $s^*=\begin{pmatrix}0\\\hat{s}_N\\\shat_{n+1}\end{pmatrix}$ and $s^0=\begin{pmatrix}s^0_B\\s_N^0\\s^0_{n+1}\end{pmatrix}$

\State Generate $y^0=\begin{pmatrix}y^0_{1:m}\\y^0_{m+1}\end{pmatrix}\in \mathbb{R}^{m+1}$ randomly such that $y^0_{m+1}\not=0$ 

\State Build $y^*=\begin{pmatrix}\hat{y}\\0\end{pmatrix}$ 

\State Calculate 
\begin{align*}
    \hat{a}_{n+1}&=\frac{1}{x^0_{n+1}}(\hat{A}_B(\hat{x}_B-x^0_B)-\hat{A}_Nx^0_N)\\
    d_B&=\frac{1}{y^0_{m+1}}(\Ahat_B^{\top} (\hat{y}-y^0_{1:m})-s^0_B)\\
    d_N&=\frac{1}{y^0_{m+1}}(\Ahat_N^{\top} (\hat{y}-y^0_{1:m})+s^*_N-s^0_N)\\
    d_{n+1}&=\frac{1}{x^0_{n+1}}(d_B^{\top} (\hat{x}_B-x_B^0)-d_N^{\top} x_N^0)
\end{align*} 
\State Build $A_{(m+1)\times(n+1)}=\begin{pmatrix}\hat{A}_B&\hat{A}_N&\hat{a}_{n+1}\\ d_B^{\top} &d_N^{\top} &d_{n+1}\end{pmatrix}$
\State Calculate $b=\begin{pmatrix}\hat{b}\\
d_B^{\top} \hat{x}_B\end{pmatrix}$ and $c=\begin{pmatrix}\hat{c}\\\hat{a}_{n+1}^{\top} \hat{y}+s^*_{n+1}\end{pmatrix}$
\State \textbf{Return }{LOP $(A, b, c)$ with optimal solution
$(x^*,y^*,s^*)$ and interior solution $(x^0,y^0,s^0)$}
\end{algorithmic}
\end{algorithm}

Theorem \ref{t1} asserts that the claimed properties of $(x^0,y^0,s^0)$ and  $(x^*,y^*,s^*)$ are indeed correct. Before presenting and proving Theorem~\ref{t1}, we need to verify the orthogonality properties of the generated solution.
\begin{lemma}\label{lem: orthogonality of LO}
For $(x^0,y^0,s^0)$ and  $(x^*,y^*,s^*)$ generated by Algorithm \ref{alg: LO-OP-IP}, then we have
$$(x^0-x^*)^{\top} (s^0-s^*)=0.$$
\end{lemma}

\begin{proof}
By construction, we have 
\begin{align*}
    (x^0-x^*)^{\top} (s^0-s^*)=&(x_B^0)^{\top}s_B^0+(x_B^*)^{\top}s_B^*-(x_B^0)^{\top}s_B^*-(x_B^*)^{\top}s_B^0(x_N^0)^{\top}s_N^0 \\
    &+(x_N^*)^{\top}s_N^*-(x_N^0)^{\top}s_N^*-(x_N^*)^{\top}s_N^0 (x_{n+1}^0)^{\top}s_{n+1}^0\\
    &+(x_{n+1}^*)^{\top}s_N^*-(x_{n+1}^0)^{\top}s_{n+1}^*-(x_{n+1}^*)^{\top}s_{n+1}^0\\
    =&-\delta+(x_{n+1}^0)^{\top}s_{n+1}^0-(x_{n+1}^0)^{\top}s_{n+1}^*=0.
\end{align*}
The proof is complete.
\end{proof}
Using Lemma~\ref{lem: orthogonality of LO}, the following theorem shows that the generated problem satisfies the desired properties.
\begin{theorem}\label{t1}
Let $(x^0,y^0,s^0)$ and  $(x^*,y^*,s^*)$ be generated by Algorithm~\ref{alg: LO-OP-IP}. Then, 
\begin{subequations}
\begin{align}
    x^*\geq0,~s^*\geq0,~x^0>0,~s^0&>0,\label{t1f}\\
    (x^*)^{\top} s^*&=0,\label{t1a}\\
    Ax^*&=b,\label{t1b}\\
    A^{\top} y^*+s^*&=c,\label{t1c}\\
    Ax^0&=b,\label{t1d}\\
    A^{\top} y^0+s^0&=c.\label{t1e}
\end{align}
\end{subequations}
That is, $(x^0,y^0,s^0)$ and $(x^*,y^*,s^*)$ and are, respectively, interior and optimal solutions of the generated LOP $(A,b,c)$.
\end{theorem}
\begin{proof}
Observe that \eqref{t1f} holds by construction.
Equality~\eqref{t1a} refers to compelmentarity of $(x^*,y^*,s^*)$, which holds due to the fact that
$${x^*}^{\top} s^*=\hat{x}_B^{\top} 0+0^{\top}\hat{s}_N+0s^*_{n+1}=0.$$
To see that equation \eqref{t1b} holds, i.e., the optimal solution satisfies primal feasibility, observe that 
$$Ax^*=\begin{pmatrix}\hat{A}_B&\hat{A}_N&\hat{a}_{n+1}\\ d_B^{\top} &d_N^{\top} &d_{n+1}\end{pmatrix}\begin{pmatrix}\hat{x}_B\\0\\0\end{pmatrix}=\begin{pmatrix}\hat{A}_B\hat{x}_B\\d_B^{\top} \hat{x}_B\end{pmatrix}=\begin{pmatrix}\hat{b}\\d_B^{\top} \hat{x}_B\end{pmatrix}=b.$$
Similarly, dual feasibility is satisfied by the optimal solution, since
$$A^{\top} y^*+s^*=\begin{pmatrix}\hat{A}_B^{\top} &d_B\\ \hat{A}_N^{\top} &d_N\\\hat{a}_{n+1}^{\top} &d_{n+1}\end{pmatrix}\begin{pmatrix}\hat{y}\\0\end{pmatrix}+\begin{pmatrix}0\\\hat{s}_N\\s^*_{n+1}\end{pmatrix}=\begin{pmatrix}\hat{c}\\\hat{a}_{n+1}^{\top} \hat{y}+s^*_{n+1}\end{pmatrix}=c.$$
That is, equation \eqref{t1c} holds. 

The interior solution $(x^0,y^0,s^0)$ is primal feasible since
\begin{align*}
Ax^0&=\begin{pmatrix}\hat{A}_B&\hat{A}_N&\hat{a}_{n+1}\\ d_B^{\top} &d_N^{\top} &d_{n+1}\end{pmatrix}\begin{pmatrix}x^0_B\\x^0_N\\x^0_{n+1}\end{pmatrix}\\
    &=\begin{pmatrix}\hat{A}_Bx^0_B+\hat{A}_Nx^0_N+\hat{a}_{n+1}x^0_{n+1}\\ d_B^{\top} x^0_B+d_N^{\top} x^0_N+d_{n+1}x^0_{n+1}\end{pmatrix}\\
    &=\begin{pmatrix}\hat{A}_Bx^0_B+\hat{A}_Nx^0_N+(\hat{A}_B(\hat{x}_B-x^0_B)-\hat{A}_Nx^0_N)\\ d_B^{\top} x^0_B+d_N^{\top} x^0_N+(d_B^{\top}(\hat{x}_B-x_B^0)-d_N^{\top} x_N^0)\end{pmatrix}\\
    &=\begin{pmatrix}\hat{A}_B\hat{x}_B\\d_B^{\top} \hat{x}_B\end{pmatrix}=\begin{pmatrix}\hat{b}\\d_B^{\top} \hat{x}_B\end{pmatrix}=b,
\end{align*}
which proves \eqref{t1d}. We can also certify the dual feasibility of the interior solution: 
\begin{align*}
    A^{\top} y^0+s^0&=\begin{pmatrix}\hat{A}_B^{\top} &d_B\\ \hat{A}_N^{\top} &d_N\\\hat{a}_{n+1}^{\top} &d_{n+1}\end{pmatrix}\begin{pmatrix}y^0_{1:m}\\y^0_{m+1}\end{pmatrix}+\begin{pmatrix}s^0_B\\s_N^0\\s^0_{n+1}\end{pmatrix}=\begin{pmatrix}\hat{A}_B^{\top} y^0_{1:m}+d_By^0_{m+1}+s^0_B\\\hat{A}_n^{\top} y^0_{1:m}+d_Ny^0_{m+1}+s_N^0\\\hat{a}_{n+1}^{\top} y^0_{1:m}+d_{n+1}y^0_{m+1}+s^0_{n+1}\end{pmatrix}\\
    &=\begin{pmatrix}\hat{A}_B^{\top} y^0_{1:m}+(\Ahat_B^{\top}(\hat{y}-y^0_{1:m})-s^0_B)+s^0_B\\\hat{A}_n^{\top} y^0_{1:m}+(\Ahat_N^{\top}(\hat{y}-y^0_{1:m})+s^*_N-s^0_N)+s_N^0\\\alpha\end{pmatrix}=\begin{pmatrix}\Ahat_B^{\top}\hat{y}\\\Ahat_N^{\top}\hat{y}+s^*_N\\\alpha\end{pmatrix}\\
    &=\begin{pmatrix}\hat{c}\\\hat{a}_{n+1}^{\top}\hat{y}+s^*_{n+1}\end{pmatrix}=c,
\end{align*}
where $\alpha= \hat{a}_{n+1}^{\top} y^0_{1:m}+d_{n+1}y^0_{m+1}+s^0_{n+1}$. 

Finally, to prove that equation \eqref{t1e} holds as well, we still need to show that $\alpha=\hat{a}_{n+1}^{\top}\hat{y}+s^*_{n+1}$. By straightforward calculation, we have

\noindent \resizebox{1\linewidth}{!}{
	\begin{minipage}{\linewidth}
 \begin{align*}
   \alpha&=\frac{1}{x^0_{n+1}}\big(\hat{A}_B(\hat{x}_B-x^0_B)-\hat{A}_Nx^0_N\big)^{\top}y^0_{1:m}+\frac{y^0_{m+1}}{x^0_{n+1}}\big(d_B^{\top}(\hat{x}_B-x_B^0)-d_N^{\top}x_N^0\big)+s^0_{n+1} \\
   &=\frac{1}{x^0_{n+1}}\Big(\big(\hat{A}_B(\hat{x}_B-x^0_B)-\hat{A}_Nx^0_N\big)^{\top}y^0_{1:m}+\big(\Ahat_B^{\top}(\hat{y}-y^0_{1:m})-s^0_B\big)^{\top}(\hat{x}_B-x_B^0)-\big(\Ahat_N^{\top}(\hat{y}-y^0_{1:m})+s^*_N-s^0_N\big)^{\top}x_N^0\Big)+s^0_{n+1} \\
   &=\frac{1}{x^0_{n+1}}\Big({y^0_{1:m}}^{\top}\hat{A}_B\hat{x}_B-{y^0_{1:m}}^{\top}\hat{A}_Bx^0_B-{y^0_{1:m}}^{\top}\hat{A}_Nx^0_N+\hat{x}_B^{\top}\Ahat_B^{\top}\hat{y}-\hat{x}_B^{\top}\Ahat_B^{\top}y^0_{1:m}-\hat{x}_B^{\top}s^0_B-{x_B^0}^{\top}\Ahat_B^{\top}\hat{y}+{x_B^0}^{\top}\Ahat_B^{\top}y^0_{1:m}+{x_B^0}^{\top}s^0_B\\
   &\quad-{x_N^0}^{\top}\Ahat_N^{\top}\hat{y}+{x_N^0}^{\top}\Ahat_N^{\top}y^0_{1:m}-{x_N^0}^{\top}s^*_N+{x_N^0}^{\top}s^0_N\Big)+s^0_{n+1} \\
   &=\frac{1}{x^0_{n+1}}\Big(\hat{x}_B^{\top}\Ahat_B^{\top}\hat{y}-{x_B^0}^{\top}\Ahat_B^{\top}\hat{y}-{x_N^0}^{\top}\Ahat_N^{\top}\hat{y}-\hat{x}_B^{\top}s^0_B+{x_B^0}^{\top}s^0_B-{x_N^0}^{\top}s^*_N+{x_N^0}^{\top}s^0_N\Big)+s^0_{n+1} \\
   &=\frac{(\hat{x}_B^{\top}\Ahat_B^{\top}-{x_B^0}^{\top}\Ahat_B^{\top}-{x_N^0}^{\top}\Ahat_N^{\top})}{x^0_{n+1}}\hat{y}+\frac{-\hat{x}_B^{\top}s^0_B+{x_B^0}^{\top}s^0_B-{x_N^0}^{\top}s^*_N+{x_N^0}^{\top}s^0_N+s^0_{n+1}x^0_{n+1}}{x^0_{n+1}} \\
   &=\hat{a}_{n+1}^{\top}\hat{y}+\frac{(x^0-x^*)^{\top}(s^0-s^*)+s^*_{n+1}x^0_{n+1}}{x^0_{n+1}}=\hat{a}_{n+1}^{\top}\hat{y}+s^*_{n+1}.
\end{align*}
	\end{minipage}
\newline}

The proof is complete.
\end{proof}
\begin{remark}
Since the generated optimal solution $(x^*,y^*,s^*)$ by Algorithm~\ref{alg: LO-OP-IP} is strictly complementary, the optimal partition $(\Bcal, \Ncal)$ is equal to $(B,N)$.
If we modify Algorithm~\ref{alg: LO-OP-IP} such that  $x_i^*\geq0$  for $i\in B$ and  $s_i^*\geq0$  for $i\in N$, then $B$ and $N$ do not necessarily give the optimal partition. While $x^*$ and $s^*$ are complementary solutions, they are not necessarily strictly complementary.
\end{remark}

\begin{remark}
We can simplify Algorithm \ref{alg: LO-OP-IP} by setting
$$x_B^0=\hat{x}_B,  s_N^0=\hat{s}_N, s^*_{n+1}=s^0_{n+1},\begin{text} and  \end{text}~ y^0_{1:m}=\hat{y}.$$ 
It is straightforward to verify that Condition~\ref{con1} is satisfied for these choices. An even simpler case arises if we choose
$$x^0_N=e, s^0_B=e, \text{ and }y^0_{m+1}=x^0_{n+1}=s^0_{n+1}=1.$$
\end{remark}

In the next section, we extend these problem generators to generate SDO problems.

\section{Semidefinte Optimization}\label{sec: SDO}
Now, we turn our attention to SDO. Just as in the previous section, we begin by reviewing the problem setting and important properties before presenting the instance generators for this class of optimization problems. 

\subsection{Semidefinte Optimization Problems}\label{sec: SDOP}
In \textit{semidefinite optimization}, one seeks to minimize the inner product
of two $n \times n$ symmetric matrices:
$$ C \bullet X = \trace{(CX)} = \sum_{i = 1}^n \sum_{j = 1}^n C_{ij} X_{ij},$$
for some symmetric constant matrix $C \in \Scal^n$ and matrix variable $X\in \Scal^n$. Note that $\Scal^n$ denotes the space of $n \times n$ symmetric matrices, and we write $\Scal^n_+$ ($\Scal^n_{++}$) to represent the cone of symmetric positive semidefinite (symmetric positive definite) matrices. 

Similar to the LOP studied in the previous section, variable $X$ must satisfy linear constraints of the form
$$A_i \bullet X = b_i,~~~\forall i \in [m],$$
where $A_1, \dots, A_m \in \Scal^n$ are given symmetric matrices and $b \in \mathbb{R}^{m}$. Given that $C \bullet X$ is a linear function of $X$, stopping here would simply yield a LOP in which the variables are given by the entries of the matrix $X$. Rather, we add a nonlinear (albeit convex) constraint, which stipulates that $X$ must be a positive semidefinite matrix, which we write $X \succeq 0$. More generally, the notation $U \succeq V$ indicates that $U - V$ is symmetric positive semidefinite, and is equivalent to stating $U - V \in \Scal^n_+$. Likewise, when the inequality is strict, i.e., $U \succ V$, it follows that $U - V \in \Scal^n_{++}$, so $U-V$ is symmetric positive definite. From the above discussion, it is straightforward to observe that SDO is a generalization of LO, in which we replace the element-wise nonnegativity constraint $x \geq 0$ found in \eqref{e:LO-P} by a conic inequality with respect to the cone $\Scal^n_+$. 

Accordingly, in this section we are interested in generating problems of the form
\begin{equation}\label{e:SDO1}
     z^P_{SDO} = \inf_X \left\{C \bullet X : A_i \bullet X = b_i,~\forall i \in [m], X \succeq 0 \right\}, \tag{SDOP-P}
\end{equation}
which has an associated dual problem
\begin{equation}\label{e:SDO2}
    z^D_{SDO} =  \sup_{(y, S)} \left\{ b^{\top} y:\sum_{i=1}^m y_i A_i + S = C,~S\succeq 0, y \in \mathbb{R}^{m} \right\},\tag{SDOP-D}
\end{equation}
where $S = C - \sum_{i=1}^m y_i A_i$ is the slack matrix of the dual problem. Without loss of generality, we may assume that the matrices $A_1, \dots, A_m$ are linearly independent. 

If $X$ and $(y,S)$ satisfy the primal and dual constraints, respectively, we say that they are feasible solutions, denoting the feasible sets of \eqref{e:SDO1} and \eqref{e:SDO2} by:
\begin{align*}
    \Pcal_{SDO} &= \left\{X \in \Scal^n  : A_i \bullet X = b_i, ~i \in [m], X \succeq 0 \right\} \\
    \Dcal_{SDO} &=  \left\{(y, S) \in \mathbb{R}^{m} \times \Scal^n \; : \sum_{i=1}^m y_i A_i + S = C, S \succeq 0 \right\}.
\end{align*}
Accordingly, the sets of feasible interior solutions are given by
\begin{align*}
    \Pcal_{SDO}^0 &= \left\{ X \in \Pcal_{SDO} :  X \succ 0 \right\},\\
    \Dcal_{SDO}^0 &= \left\{ (y, S) \in \Pcal_{SDO} :  S \succ 0 \right\}.
\end{align*}
For ease of notation, we adopt the syntax $\Pcal \Dcal_{SDO}=\Pcal_{SDO}\times\Dcal_{SDO}$ and $\Pcal \Dcal_{SDO}^0=\Pcal_{SDO}^0\times\Dcal_{SDO}^0$.

Just as in the case of LO, when IPMs are applied to SDOPs, it is standard to assume the existence of a strictly feasible primal-dual pair $X$ and $(y, S)$ with $(X, S) \succ 0$. From the existence of a strictly feasible initial solution $(X^0, S^0) \succ 0$, it follows that the Interior Point Condition (IPC) is satisfied \citep{de2006aspects}, guaranteeing that the primal and dual optimal sets
\begin{align*}
    \Pcal_{SDO}^* &= \left\{X \in \Pcal_{SDO} :  C \bullet X = z^P_{SDO} \right\}, \\
    \Dcal_{SDO}^* &=  \left\{(y, S) \in \Dcal_{SDO} :  b^{\top} y = z^D_{SDO} \right\},
\end{align*}
are nonempty and bounded, that an optimal primal-dual pair with zero duality gap exists, i.e., strong duality holds. That is, for optimal solutions
$( X^*, y^*, S^*) \in \Pcal\Dcal_{SDO}^*$, where $\Pcal \Dcal_{SDO}^*=\Pcal_{SDO}^*\times\Dcal_{SDO}^*$, we have
\begin{equation*}
    C \bullet X^* - b^{\top} y^*  = X^* \bullet S^*  = 0,
\end{equation*}
which implies $X^* S^* = S^* X^* = 0$ as $X^*$ and $S^*$ are symmetric
positive semidefinite matrices.

\subsection{Instance Generators for SDOPs}\label{sec: SDOGen}

Similar to our work on LO, we propose three generators that produce SDO instances with a predefined interior solution, optimal solution, and both. Each generator is designed such that the user can control the characteristics of parameters such as condition number, sparsity, matrix structure, and size. Additionally, users can modify the features of optimal solutions to evaluate the performance of their algorithms.

\subsubsection{SDOPs with a Predefined Interior Solution}\label{sec: SDOINT}

To study the performance of IPMs applied to SDO, it is helpful to have instances with a specific interior solution. Generally, some users may need to generate problems with an interior solution to ensure that Strong Duality, i.e., zero duality gap at optimality, holds. Along this line, we adapt Algorithm~\ref{alg: LO-IP} to generate SDO instances with known interior solutions, as given in Algorithm~\ref{alg: SDO-IP}. 

\begin{algorithm}[H] 
\caption{Generating SDO problems with a specific interior solution} \label{alg: SDO-IP}
\begin{algorithmic}[1]
\State Choose dimensions $m,n$ with $m<\frac{n(n+1)}{2}$
\State Generate $(X^0,S^0)$ such that $X^0\succ0$ and $S^0\succ0$ 
\State Generate $A_i \in \mathcal{S}^n$ for $i \in [m]$
\State Generate $y^0 \in \mathbb{R}^{m}$   
\State Calculate $b_i= A_i \bullet X^0$ for $i \in [m]$ and $C=\sum_{i=1}^{m}y^0_iA_i+S^0$
\State \textbf{Return }{SDOP $(A_1, \dots, A_m, b, C)$ with interior solution
$(X^0,y^0,S^0)$}
\end{algorithmic}
\end{algorithm}

Compared to Algorithm~\ref{alg: LO-IP}, the task of generating $X^0$ and $S^0$ in a general manner such that $X^0S^0=\mu I$ for $\mu>0$ is more computationally involved; we would first have to generate $X^0\succ0$ randomly, and subsequently calculate $S^0$ as $S^0=\mu (X^0)^{-1}$. However, we can easily generate $X^0$ and $S^0$ for a specified value of $\mu$ if we make additional assumptions regarding their structure (e.g., we can assume they are diagonal). We can also generate the matrices $A_1, \dots, A_m$ to have desired properties such as sparsity, conditioning, or to satisfy some norm bound. Several approaches for generating random positive semidefinite are discussed in Appendix~\ref{appendix:PSD}.

\subsubsection{SDOPs with a Predefined Block-diagonal Optimal Solution}\label{sec: SDOOPT1}
Algorithm~\ref{alg: SDO-OP} can be seen as a generalization of Algorithm~\ref{alg: LO-OP} to SDO problems, in which the generated optimal solution explicitly has a block-diagonal structure corresponding to the optimal partition. Before presenting the instance generator, we review the notation of the optimal partition in the context of SDO. 
 
 We are interested in problems whose optimal solution $(X^*,y^*, S^*)$ exhibits zero duality gap, i.e., $X^*S^*=0$. Thus, the spectral decomposition of an optimal pair $X^*$ and $S^*$ takes the form
 $$X^*=Q\Sigma Q^{\top} \text{ and } S^*=Q\Lambda Q^{\top},$$
 where $Q$ is orthonormal, and the matrices $\Sigma$ and $\Lambda$ are diagonal, containing eigenvalues of $X^*$ and $S^*$, respectively. Letting $\sigma_i=\Sigma_{i,i}$ and $\lambda_i=\Lambda_{i,i}$, it follows that $X^*S^*=0$ holds if and only if $\sigma_i\lambda_i=0$ for all $i\in [n]$. 
 A primal-dual optimal solution $(X^*,y^*,S^*) \in {\Pcal}{\Dcal}_{SDO}^*$ is called maximally complementary if $X^* \in \ri({\Pcal}_{SDO}^*)$ and $(y^*,S^*) \in \ri({\Dcal}_{SDO}^*)$. A maximally complementary optimal solution $(X^*,y^*,S^*)$ is called strictly complementary if $X^* + S^* \succ 0$. 
 Let $\Bcal \coloneqq \Rcal(X^*)$ and $\Ncal \coloneqq \Rcal(S^*)$, where $(X^*,y^*,S^*)$ is a maximally complementary optimal solution and $\Rcal(.)$ denotes the range space. We define $n_{\Bcal} \coloneqq \dim (\Bcal)$ and $n_{\Ncal} \coloneqq \dim (\Ncal)$. 
 Then, we have $\Rcal(X) \subseteq B$ and $\Rcal(S) \subseteq \Ncal$ for all $(X,y,S) \in \Pcal\Dcal_{SDO}^*$. 
 By the complementarity condition, the subspaces $\Bcal$ and $\Ncal$ are orthogonal, and this implies that $n_{\Bcal} + n_{\Ncal} \leq n$, and in case of strict complementarity, $n_{\Bcal} + n_{\Ncal} = n$. Otherwise, a subspace $\Tcal$ exists, which is the orthogonal complement to $\Bcal+\Ncal$. Similarly, we have $n_{\Tcal} \coloneqq \dim (\Tcal)$, and so $n_{\Bcal}+n_{\Ncal}+n_{\Tcal}=n$ \citep{mohammad2020identification}.
The partition $(\Bcal, \Ncal, \Tcal)$ of $\Rmbb^n$ is called the optimal partition of an SDO problem. In LOPs, we know that $\Tcal$ is empty, but in general SDOPs $\Tcal$ can be non-empty \citep{de2006aspects}. 

In Algorithm~\ref{alg: SDO-OP}, we generate SDOPs with optimal solutions which exhibit a block-diagonal structure using a partition $(B,N,T)$, which may be different from the optimal partition $(\Bcal,\Ncal,\Tcal)$ of the generated problem.

\begin{algorithm}[H]
\caption{Generating SDO problems with a specific optimal solution} \label{alg: SDO-OP}
\begin{algorithmic}[1]
\State Choose dimensions $m,n$ with $m<\frac{n(n+1)}{2}$
\State Choose $n_{B},n_{N}\in[n]$ where $n_{B}+n_{N}\leq n$
\State Generate  positive definite matrix $X_B \in \Scal_{++}^{n_{B}}$
\State Generate positive definite matrix $S_N\in \Scal_{++}^{n_{N}}$
\State Build\footnote{The matrices are partitioned according to $n_B$, $n_T$, and $n_N$.} $X^*=\begin{pmatrix}
X_B &0& 0\\
0 & \text{0}&0\\
0&0&0
\end{pmatrix}$ and  $S^*=\begin{pmatrix}
0 & 0&0\\
0 & \text{0}&0\\
0 & 0&S_N
\end{pmatrix}$
\State Generate $A_i \in \mathcal{S}^n$ for $i \in [m]$
\State Generate $y^* \in \mathbb{R}^{m}$  
\State Calculate $b_i=A_i \bullet X^*$ for $i \in [m]$ and $C=\sum_{i=1}^{m}y^*_iA_i+S^*$
\State \textbf{Return }{SDOP $(A_1,\dots,A_m, b, C)$ with optimal solution
$(X^*,y^*,S^*)$}
\end{algorithmic}
\end{algorithm}

\begin{remark}\label{rem:notmax}
The sets $(B,N,T)$ generated in Algorithm~\ref{alg: SDO-OP} are not necessarily the optimal partition $(\Bcal,\Ncal,\Tcal)$ for the generated SDO problem $(A_1,\dots,A_m, b, C)$. In general, we only have
\begin{equation*}
    B\subseteq\Bcal, N\subseteq\Ncal,\text{ and } \Tcal\subseteq T.
\end{equation*}
\end{remark}
\begin{remark}
If an SDOP with a strictly complementary optimal solution is required, then we set $n_{N}=n-n_{B}$. In this case the optimal partition is predefined as $\Bcal=B$, $\Ncal=N$, and $\Tcal=\emptyset$.
\end{remark}
One can easily verify that the solution $(X^*,y^*,S^*)$ generated by Algorithm~\ref{alg: SDO-OP} is feasible for the SDO problem $(A_1,\dots,A_m,b,C)$, and optimal since $X^*S^*=0$. In addition,  matrices $A_i$ and $C$ can be generated in a way to exhibit a particular sparsity, condition number or norm, and we can also control primal and/or dual degeneracy. 

\subsubsection{SDOPs with Predefined Block-diagonal Optimal and Interior Solutions}\label{sec: SDOINTOPT1}

By generating SDO problems with specific interior and optimal solutions, we can study the performance of various solution approaches. For example, one can analyze how efficiently feasible IPMs reduce the complementarity starting from a predefined interior solution to an optimal solution, or alternatively examine how robust performance is to the provided starting point or changes in the characteristics of the optimal solutions or partition. To accomplish this, we propose several algorithms in this paper providing an optimal solution or a maximally complementary solution.

This section is focused on the case in which the user is interested in predefining an interior solution and an optimal solution, which need not necessarily be maximally complementary. Accordingly, Algorithm~\ref{alg: SDO-IP-OP} generalizes Algorithm~\ref{alg: LO-OP-IP} to SDO for the case in which the generated optimal solution has a block-diagonal structure. We similarly seek to generate an optimal solution $(X^*,y^*, S^*)$ and interior solution $(X^0, y^0, S^0)$ as generally as possible, but we need to impose some additional requirements. Letting $\Lcal=\text{span}\{A_1,\dots,A_m\}$, we have $X^0-X^*\in \Lcal^{\perp}$ and $S^0-S^*\in \Lcal$, and hence, the generated solutions are required to satisfy the orthogonality condition
\begin{equation}\label{con2}
    (X^0-X^*) \bullet (S^0-S^*) =0.
\end{equation}
In Algorithm~\ref{alg: SDO-IP-OP}, steps~\ref{al7s12} and \ref{al7s13} are designed to ensure the generated solutions $(X^*,y^*,S^*)$ and $(X^0, y^0, S^0)$ indeed satisfy orthogonality.
\begin{algorithm}[H]

\caption{Generating SDO problems with specific interior and optimal solutions} \label{alg: SDO-IP-OP}
\begin{algorithmic}[1]
\State Choose dimensions $m,n$ with $m<\frac{n(n+1)}{2}$
\State Choose $n_B,n_N\in [n]$ where $n_B+n_N\leq n$
\State Generate SDO problem $(\Ahat_1,\dots ,\Ahat_m, \bhat,\Chat)$ with optimal solution $(\Xhat,\yhat,\shat)$ using Algorithm~\ref{alg: SDO-OP}

\State Generate $X^0_B\succ0$, $X^0_T\succ0$,$X^0_N\succ0$,$X^0_{n+1}>0$ randomly
\State Build $X^*_{(n+1)\times(n+1)}=\begin{pmatrix}
\hat{X}& 0\\
0 & 0
\end{pmatrix}$ and $X^0_{(n+1)\times(n+1)}=\begin{pmatrix}
X_B^0&0& 0&0\\
0&X^0_T&0&0\\
0 &0& X^0_N&0\\
0&0&0&X^0_{n+1}
\end{pmatrix}$
\State Generate $S^0_T\succ0$, $S^0_B\succ0$, $S^0_N\succ0$ randomly\label{al7s11}
\State Calculate $\delta=(X_B^0-\Xhat_B)\bullet S_B^0+X_T^0\bullet S_T^0+X_N^0\bullet(S_N^0-\Shat_N)$\label{al7s12}
\State Generate $S_{n+1}^0>(\frac{-\delta}{X^0_{n+1}})^+$ and calculate $\hat{S}_{n+1}=\frac{\delta}{X^0_{n+1}}+S_{n+1}^0$\label{al7s13}
\State Build $S^*_{(n+1)\times(n+1)}=\begin{pmatrix}
\hat{S}& 0\\
0 & \hat{S}_{n+1}
\end{pmatrix}$ and $S^0_{(n+1)\times(n+1)}=\begin{pmatrix}
S^0_B&0& 0&0\\
0&S^0_T&0&0\\
0&0& S^0_N&0\\
0&0&0&S^0_{n+1}
\end{pmatrix}$
\State Generate $y^0\in \mathbb{R}_{m+1}$ randomly such that $y_{m+1}^0\not=0$

\State Build $y^*=\begin{pmatrix}
\hat{y}\\
0 
\end{pmatrix}\in \mathbb{R}_{m+1}$ 
\State Calculate
$
    \alpha_i=\frac{1}{X_{n+1}^0}( \hat{A}_{i_B} \bullet (X_B-X^0_B))-(\hat{A}_{i_N} \bullet X^0_N)- (\hat{A}_{i_T} \bullet X^0_T)) 
$ for $i \in [m]$
\State Build $A_i= \begin{pmatrix}
\hat{A}_i& 0\\
0 &\alpha_i
\end{pmatrix}$  for $i \in [m]$
\State Build $\displaystyle A_{n+1}=\sum_{i=1}^{m}\frac{\hat{y}_i-y^0_i}{y^0_{m+1}}A_i+  \frac{1}{y^0_{m+1}}\begin{pmatrix}
-S^0_B&0& 0&0\\
0 &-S^0_T& 0&0\\
0&0&\hat{S}_N-S^0_N&0\\
0&0&0&\hat{S}_{n+1}-S^0_{n+1}
\end{pmatrix}$ 
\State Calculate $\theta= \hat{S}_{n+1}+\sum_{i=1}^{m}\hat{y}_i\alpha_i$ and build $C= \begin{pmatrix}
\hat{C}& 0\\
0 & \theta
\end{pmatrix}$  

\State Calculate $\beta= A_{n+1} \bullet X^*$ and build $b=\begin{pmatrix}
\hat{b}\\
\beta
\end{pmatrix}$
\State \textbf{Return }{SDOP $(A_1,\dots,A_m, b, C)$ with optimal solution $(X^*,y^*,S^*)$ and interior solution $(X^0,y^0,S^0)$}  

\end{algorithmic}
\end{algorithm}
Before proving the correctness of Algorithm \ref{alg: SDO-IP-OP}, the next result certifies the orthogonality properties of the generated solution.

\begin{lemma}\label{lem: orthogonality of SDO}
For any $(X^0,y^0,S^0)$ and $(X^*,y^*,S^*)$ generated by Algorithm \ref{alg: SDO-IP-OP}, we have
$$(X^0-X^*)\bullet(S^0-S^*)=0.$$
\end{lemma}
\begin{proof}
Similar to the proof of Lemma~\ref{lem: orthogonality of LO}, it can be proved by substitution and using Steps~\ref{al7s12} and \ref{al7s13}.
\end{proof}

Using Lemma~\ref{lem: orthogonality of SDO}, the following theorem shows that the generated problem satisfies the desired properties.
\begin{theorem} \label{theo: SDO-IP-OP}
Let $(X^0, y^0, S^0)$ and $(X^*, y^*, S^*)$ be solutions generated by Algorithm~\ref{alg: SDO-IP-OP}. Then,
\begin{subequations}
\begin{align}
X^*\succeq0,~S^*\succeq0,~X^0\succ0,~S^0&\succ0,\\
    X^* \bullet S^* &=0,\\
    A_i \bullet X^* &=b_i,\\
    \sum_{i=1}^{n}y^*_iA^{\top}_i+S^*&=C,\\
    A_i \bullet X^0 &=b_i,\label{theo2e14}\\
    \sum_{i=1}^{n}y^0_iA^{\top}_i+S^0&=C. 
\end{align}
\end{subequations}
\end{theorem}
\begin{proof}
Just as in the case of LO, all parts of Theorem~\ref{theo: SDO-IP-OP} are easy to verify based on the steps of Algorithm~\ref{alg: SDO-IP-OP}, save for equation~\eqref{theo2e14} for $i=n+1$. Following the proof of Theorem \ref{t1}, the claimed result follows from the definition of $\alpha$ and equation \eqref{con2}.
\end{proof}

Similar to generating SDOPs with an optimal solution, we can also generate problems with both a specific strictly complementary optimal solution and a specific interior solution.
\begin{remark}
One special case is when $n_B+n_N=n$, $T=\emptyset$, and
$$X^0_B=\Xhat_B, X^0_N=I, X^0_T=I,S^0_B=I,S^0_T=I,S^0_N=\Shat_N, X^0_{n+1}=1, S^0_{n+1}=1.$$
\end{remark}

\subsubsection{SDOPs with Predefined Optimal Solution (General Structure)}\label{sec: SDOOPT2}
We are also interested in the situation where the desired optimal solution does not exhibit a block structure. Some methods can exploit the structural properties of the optimal solution, for example, when it exhibits a block-diagonal structure or is sparse. In order to generate an optimal solution that possesses certain desired qualities, we use the inverse process of eigenvalue decomposition. First, we generate diagonal matrices $\Sigma$ and $\Lambda$, whose diagonal elements are the eigenvalues of $X^*$ and $S^*$, respectively. Then, $X^*$ and $S^*$ can be calculated by masking these diagonal matrices using a randomly generated orthonormal matrix $Q$, and techniques for generating orthonormal matrices are discussed in Appendix \ref{appendix:ortho}. The overall scheme is formalized below in Algorithm~\ref{alg: SDO-OP-Q}. 
\begin{algorithm}[H]

\caption{Generating SDO problems with a specific optimal solution} \label{alg: SDO-OP-Q}
\begin{algorithmic}[1]
\State Choose dimensions $m,n$ with $m<\frac{n(n+1)}{2}$
\State Choose $n_B,n_N\in[n]$ where $n_B+n_N\leq n$
\State Generate $\sigma_i>0$ for $i\in [n_B]$
and $\lambda_i>0$ for $i\in [n_N]$
\State Generate orthonormal matrix Q
\State Build $X^*=Q\begin{pmatrix}
\text{diag}(\sigma) &0& 0\\
0 & 0&0\\
0&0&0
\end{pmatrix}Q^{\top}$ and  $S^*=Q\begin{pmatrix}
0 & 0&0\\
0 & 0&0\\
0 & 0&\text{diag}(\lambda)
\end{pmatrix}Q^{\top}$
\State Generate $A_i \in \mathcal{S}^n$ randomly for $i \in [m]$
\State Generate $y^* \in \mathbb{R}^{m}$ randomly  
\State Calculate $b_i= A_i \bullet X^*$ for $i \in [m]$ and $C=\sum_{i=1}^{m}y^*_iA_i+S^*$
\State \textbf{Return }{SDOP $(A_1,\dots,A_m, b, C)$ with optimal solution
$(X^*,y^*,S^*)$}
\end{algorithmic}
\end{algorithm}
While Algorithm~\ref{alg: SDO-OP-Q} generates a more general optimal solution than Algorithm~\ref{alg: SDO-OP}, it is computationally more demanding due to several matrix multiplications and generating an orthonormal matrix. In Appendix~\ref{appendix:ortho}, some procedures to generate an orthogonal matrix are discussed. One can easily verify that $(X^*,y^*,S^*)$ is optimal since $X^*S^*=Q\Sigma\Lambda Q^{\top}=0$. However, similar to Remark~\ref{rem:notmax}, the generated optimal solution may not be the maximally complementary solution for the SDOP $(A_1,\dots,A_m, b, C)$. Thus, the optimal partition $(\Bcal, \Ncal, \Tcal)$ of the generated SDOP may be different from $(B, N, T)$ such that $\dim(\Bcal)\geq n_B$ and $\dim(\Ncal)\geq n_N$, i.e. the set of indices such that $\sigma_i=\lambda_i=0$ may be bigger than the partition $T$. The next section discusses how we can generate SDOPs with predefined optimal partition.

 \subsubsection{SDOPs with a Predefined Maximally Complementary Solution (General Structure)}\label{sec: SDOOPT2M}
The SDOP generated by Algorithm~\ref{alg: SDO-OP-Q} may have an optimal partition that differs from the input partition, since the specified optimal solution may not be maximally complementary. In this section, we develop a procedure to generate SDOPs with a specific optimal partition, and by extension, a specific maximally complementary solution.
\begin{algorithm}[H]

\caption{Generating SDO problems with a specific maximally complementary solution} \label{alg: SDO-OP-Q-M}
\begin{algorithmic}[1]
\State Choose dimensions $m,n$ with $m<\frac{n(n+1)}{2}$
\State Choose $n_B,n_N\in[n]$ where $n_B+n_N\leq n$
\State Generate $\sigma_i>0$ for $i\in [n_B]$
and $\lambda_i>0$ for $i\in [n_N]$
\State Generate orthonormal matrix Q
\State Build $X^*=Q\begin{pmatrix}
\text{diag}(\sigma) &0& 0\\
0 & 0&0\\
0&0&0
\end{pmatrix}Q^{\top}$ and  $S^*=Q\begin{pmatrix}
0 & 0&0\\
0 & 0&0\\
0 & 0&\text{diag}(\lambda)
\end{pmatrix}Q^{\top}$
\State Generate $A_1 = Q \Gamma Q^{\top}$ such that $\Gamma=\text{diag}(\gamma)$ where $\gamma_B=0$, $\gamma_T>0$, and $\gamma_N\in\Rmbb^{n_N}$
\State Generate $A_i \in \mathcal{S}^n$ $i \in [m]$ such that $A_iQ_B$ are linearly independent for $i \in [m]$
\State Generate $y^* \in \mathbb{R}^{m}$ 
\State Calculate $b_i=A_i \bullet X^*$ for $i \in [m]$ and $C=\sum_{i=1}^{m}y^*_iA_i+S^*$
\State \textbf{Return }{SDOP $(A_1,\dots,A_m, b, C)$ with maximally complementary solution
$(X^*,y^*,S^*)$}
\end{algorithmic}
\end{algorithm}
As we can see, there is less freedom in generating an SDOP using Algorithm~\ref{alg: SDO-OP-Q-M} when compared to Algorithm~\ref{alg: SDO-OP-Q}. This can be attributed to the fact that the matrix $A_1$ is specified to ensure that the specified optimal solution is maximally complementary, and we can not alter its characteristics directly. The next theorem proves the correctness of the generator.
\begin{theorem}\label{theo: maxcomSDOeigopt}
For the generated problem $(A_1,\dots,A_m, b, C)$ by Algorithm~\ref{alg: SDO-OP-Q-M}, the solution $(X^*,y^*,S^*)$ is a maximally complementary optimal solution.
\end{theorem}
\begin{proof}
The result follows from a proof by contradiction, which is adapted from \citep{wei2010generating}. Suppose that $(X^*,y^*,S^*)$ is not maximally complementary, and $(\Xtilde,\ytilde,\Stilde)$ is a maximally complementary solution. Since $\Xtilde S^*=0$, we have
$$\Rcal (X^*)\subseteq \Rcal (\Xtilde) \subseteq \Rcal(S^*)^\perp.$$
Therefore, we can write 
$$\Xtilde=Q\begin{pmatrix}D_B&0&0\\0&D_T&0\\0&0&0\end{pmatrix}Q^{\top}.
$$
Since both $\Xtilde$ and $X^*$ are feasible, it follows
\begin{align*}
    0=A_1 \bullet (\Xtilde-X^*)=A_1 \bullet Q\begin{pmatrix}D_B-\Lambda_B&0&0\\0&D_T&0\\0&0&0\end{pmatrix}Q^{\top}&=\Gamma \bullet \begin{pmatrix}D_B-\Lambda_B&0&0\\0&D_T&0\\0&0&0\end{pmatrix}\\
    &=\Gamma_T \bullet D_T.
\end{align*}
Given that $\Gamma_T>0$, it follows that $D_T=0$, which implies $\Rcal(X^*)=\Rcal(\Xtilde)$. 

Next, we need to show that $\Rcal(S^*)=\Rcal(\Stilde)$. Again, from dual feasibility, we have
$$\sum_{i=1}^{m}A_i(y^*_i-\ytilde_i)=-(S^*-\Stilde).
$$
By the orthogonality of $Q_B$ and $S^*-\Stilde$, one can observe
$$\sum_{i=1}^{m}A_iQ_B(y^*_i-\ytilde_i)=-(S^*-\Stilde)Q_B=0.
$$
Since the matrices $A_iQ_B$ are linearly independent for $i \in [m]$, it follows $y^*_i=\ytilde_i$ and $S^*=\Stilde$ and thus $(X^*,y^*,S^*)$ is maximally complementary. Therefore, we have arrived at a contradiction, and the proof is complete.
\end{proof}
\begin{corollary}
For the SDOP generated by Algorithm~\ref{alg: SDO-OP-Q-M}, the optimal partition $(\Bcal,\Ncal,\Tcal)$ is equal to $(B,N,T)$.
\end{corollary}
\begin{remark}
Let $\Pi_i=A_iQ_B$ for $i \in [m]$. One can generate matrix $\Pi_i$ for $i \in \{2,\dots,m\}$ so that the set of matrices $\Pi_i$ for $i\in[m]$ are linearly independent, and calculate $A_i=\Pi_i Q_B^{\top}$. Consequently, the matrices $A_i Q_B$ will be linearly independent with probability 1.
\end{remark}
The framework we have described is correct when $B\not = \emptyset$ and $N\not = \emptyset$. For the cases $B = \emptyset$ and/or  $N = \emptyset$, one can construct a simple procedure, such as the one presented in  Algorithm~\ref{alg: SDO-MC-Q-B}, to generate problems with predetermined optimal partition.

\begin{algorithm}[H]

\caption{Generating SDO problems with a specific maximally complementary solution when $B=\emptyset$} \label{alg: SDO-MC-Q-B}
\begin{algorithmic}[1]
\State Choose dimensions $m,n$ with $m<\frac{n(n+1)}{2}$
\State Choose $n_N\in[n]$
\State Generate $\lambda_i>0$ for $i\in [n_N]$
\State Generate orthonormal matrix Q
\State Build $X^*=0$ and  $S^*=Q\begin{pmatrix}
 0&0\\
 0&\text{diag}(\lambda)
\end{pmatrix}Q^{\top}$
\State Generate $A_i = Q \Gamma_i Q^{\top}$ such that $\Gamma_i=\text{diag}(\gamma_i)$ where  ${\gamma_T}_i=0$, and ${\gamma_N}_i\in\Rmbb^{n_N}$ for $i \in [m]$
\State Generate $y^* \in \mathbb{R}^{m}$ 
\State Let $b=0$ and calculate $C=\sum_{i=1}^{m}y^*_iA_i+S^*$
\State \textbf{Return }{SDO problem $(A_1,\dots,A_m, b, C)$ with optimal solution
$(X^*,y^*,S^*)$}
\end{algorithmic}
\end{algorithm}

\subsubsection{SDOPs with Predefined Optimal and Interior Solutions (General Structure)}\label{sec: SDOINTOPT2}
We can also generalize Algorithm~\ref{alg: SDO-OP-Q} to provide SDOPs with interior solutions, and the resulting scheme is presented in Algorithm~\ref{alg: SDO-IP-OP-Q}. Here, both the generated optimal and interior solutions have general structure by using inverse of eigenvalue decomposition and at a high level the overall scheme can be viewed as a combination of Algorithms~\ref{alg: SDO-IP-OP} and~\ref{alg: SDO-OP-Q}.

\begin{algorithm}[H]

\caption{Generating SDO problems with specific interior and optimal solutions} \label{alg: SDO-IP-OP-Q}
\begin{algorithmic}[1]
\State Choose dimensions $m,n$ with $m<\frac{n(n+1)}{2}$ (the dimensions of generated SDOP: $m+1,n+1$)
\State Choose $n_B,n_N\in[n]$ where $n_B+n_N\leq n$
\State Define sets $$B=\{1,\dots,n_B\}, T=\{n_B+1,\dots,n-n_N\}, \text{ and } N=\{n-n_N+1,\dots,n\}$$
\State Generate $\sigma_i>0$ for $i\in B$ and build $\Sigma_B=\text{diag}(\sigma)$
\State Generate $\lambda_i>0$ for $i\in N$ and build $\Lambda_N=\text{diag}(\lambda)$
\State Generate orthonormal matrix $\hat{Q}_{n\times n}$
\State Build $\hat{X}=\hat{Q}\begin{pmatrix}
\Sigma_B &0& 0\\
0 & 0&0\\
0&0&0
\end{pmatrix}\hat{Q}^{\top}$ and  $\hat{S}=\hat{Q}\begin{pmatrix}
0 & 0&0\\
0 & 0&0\\
0 & 0&\Lambda_N
\end{pmatrix}\hat{Q}^{\top}$
\State Generate $\hat{y} \in \mathbb{R}^{m}$ and $\hat{A}_i \in \mathcal{S}^n$ for $i \in [m]$ 

\State Calculate $\hat{b}_i= \hat{A}_i \bullet X^*$ for $i \in [m]$ and $\hat{C}=\sum_{i=1}^{m}\hat{y}_i\hat{A}_i+\hat{S}$

\State Build $Q_{(n+1)\times(n+1)}=\begin{pmatrix}
\hat{Q}& 0\\
0 & 1
\end{pmatrix}$
\State Generate positive diagonal matrix $\Sigma_B^0$, $\Sigma_T^0$, $\Sigma_N^0$, and number $\sigma_{n+1}^0>0$ 
\State Build $X^*=Q\begin{pmatrix}
\Sigma_B&0& 0&0\\
0&0&0&0\\
0 &0& 0&0\\
0&0&0&0
\end{pmatrix}Q^{\top}$ and $X^0=Q\begin{pmatrix}
\Sigma_B^0&0& 0&0\\
0&\Sigma_T^0&0&0\\
0 &0& \Sigma_N^0 &0\\
0&0&0&\sigma_{n+1}^0
\end{pmatrix}Q^{\top}$
\State Generate positive diagonal matrix $\Lambda_B^0$, $\Lambda_T^0$, and $\Lambda_N^0$ 
\State Calculate $\delta=\sum_{i\in B}(\sigma_i-\sigma_i^0)\lambda_i^0+\sum_{i\in T}\sigma_i^0\lambda_i^0+\sum_{i\in N}\sigma_i^0(\lambda_i^0-\lambda_i)$\label{alg10s13}
\State Generate $\lambda^0_{n+1}>(\frac{-\delta}{\sigma_{n+1}^0})^+$, and calculate $\lambda_{n+1}=\frac{\delta}{\sigma_{n+1}^0}+\lambda^0_{n+1}$ \label{alg10s14}
\State Build $S^*=Q\begin{pmatrix}
0&0& 0&0\\
0&0&0&0\\
0&0& \Lambda_N&0\\
0&0&0&\lambda_{n+1}
\end{pmatrix}Q^{\top}$ and $S^0=Q\begin{pmatrix}
\Lambda_B^0&0& 0&0\\
0&\Lambda_T^0&0&0\\
0&0& \Lambda_N^0&0\\
0&0&0&\lambda_{n+1}^0
\end{pmatrix}Q^{\top}$
\State Generate $y^0\in \mathbb{R}^{m+1}$ randomly such that $y^0_{m+1}\not=0$ 
and build $y^*_{(m+1)}=\begin{pmatrix}
\hat{y}\\
0 
\end{pmatrix}$ 
\State Calculate $\alpha_i=\frac{1}{\sigma_{n+1}^0}\trace{\left(\hat{A}_i\hat{Q}\begin{pmatrix}
\Sigma_B-\Sigma_B^0&0& 0\\
0&-\Sigma_T^0&0\\
0 &0& -\Sigma_N^0 
\end{pmatrix}\hat{Q}^{\top}\right)}$ for $i \in [m]$
\State Build $A_i= \begin{pmatrix}
\hat{A}_i& 0\\
0 & \alpha_i
\end{pmatrix}$  for $i \in [m]$
\State $A_{m+1}= \frac{1}{y^0_{m+1}}\left(\sum_{i=1}^{m}(\yhat_i-y^0_i)\Ahat_i+Q\begin{pmatrix}
-\Lambda_B^0&0& 0&0\\
0&-\Lambda_T^0&0&0\\
0&0&\Lambda_T -\Lambda_T^0&0\\
0&0&0&\Lambda_{n+1} -\Lambda_{n+1}^0
\end{pmatrix}Q^{\top} \right)$
\State Calculate $C=\begin{pmatrix}
\hat{C}& 0\\
0 & \sum_{i=1}^{m}\hat{y}_i\alpha_i+\lambda_{m+1}
\end{pmatrix}$

\State Calculate  $b_i=\hat{b}_i$ for $i\in [m]$ and $b_{m+1}=\trace{(A_{m+1}X^*)}$
\State \textbf{Return }{SDOP $(A_i, b, C)$ with optimal solution $(X^*,y^*,S^*)$ and interior solution $(X^0,y^0,S^0)$}  

\end{algorithmic}
\end{algorithm}
For the generated solutions $(X^0,y^0,S^0)$ and  $(X^*,y^*,S^*)$, Steps \ref{alg10s13} and \ref{alg10s14} ensure that
\begin{equation}
    \sum_{i\in B}(\sigma_i^0-\sigma_i)\lambda_i^0+\sum_{i\in T}\sigma_i^0\lambda_i^0+\sum_{i\in N}\sigma_i^0(\lambda_i^0-\lambda_i)+\sigma_{n+1}^0(\lambda_{n+1}^0-\lambda_{n+1})=0,
\end{equation}
Consequently, the orthogonality condition \eqref{con2} is satisfied. 
Similar to the block-diagonal case, Theorem~\ref{theo: SDO-IP-OP-Q} establishes that the generated SDO problem and its optimal and interior solutions are correct.
\begin{theorem}\label{theo: SDO-IP-OP-Q}
Let $(X^0, y^0, S^0)$ and $(X^*, y^*, S^*)$ be solutions generated by  Algorithm~\ref{alg: SDO-IP-OP-Q}. Then,
\begin{subequations}
\begin{align*}
X^*\succeq0,~S^*\succeq0,~X^0\succ0,~S^0&\succ0,\\
    X^* \bullet S^*&=0,\\
    A_i \bullet X^*&=b_i,\\
    \sum_{i=1}^{n}y^*_iA^{\top}_i+S^*&=C,\\
    A_i \bullet X^0 &=b_i,\\
    \sum_{i=1}^{n}y^0_iA^{\top}_i+S^0&=C.
\end{align*}
\end{subequations}
\end{theorem}

\begin{proof}
Analogous to the proof of Theorem~\ref{theo: SDO-IP-OP} based on the steps of Algorithm~\ref{alg: SDO-IP-OP-Q} and Condition \eqref{con2}.
\end{proof}
\subsubsection{SDOPs with Predefined Interior and Maximally Complementary Solutions (General Structure)}\label{sec: SDOINTOPT3}
To have a predetermined optimal partition, we develop Algorithm~\ref{alg: SDO-IP-OP-Q-M} to generate SDOPs with specific interior and maximally complementary solutions as follows.
\begin{algorithm}[H]

\caption{Generating SDO problems with specific interior and maximally complementary solutions} \label{alg: SDO-IP-OP-Q-M}
\begin{algorithmic}[1]
\State Choose dimensions $m,n$ with $m<\frac{n(n+1)}{2}$ (the dimensions of generated SDOP: $m+1,n+1$)
\State Choose $n_B,n_N\in[n]$ where $n_B+n_N\leq n$
\State Define sets $$B=\{1,\dots,n_B\}, T=\{n_B+1,\dots,n-n_N\}, \text{ and } N=\{n-n_N+1,\dots,n\}$$
\State Generate $\sigma_i>0$ for $i\in B$ and build $\Sigma_B=\text{diag}(\sigma)$
\State Generate $\lambda_i>0$ for $i\in N$ and build $\Lambda_N=\text{diag}(\lambda)$
\State Generate orthonormal matrix $\hat{Q}_{n\times n}$
\State Build $\hat{X}=\hat{Q}\begin{pmatrix}
\Sigma_B &0& 0\\
0 & 0&0\\
0&0&0
\end{pmatrix}\hat{Q}^{\top}$ and  $\hat{S}=\hat{Q}\begin{pmatrix}
0 & 0&0\\
0 & 0&0\\
0 & 0&\Lambda_N
\end{pmatrix}\hat{Q}^{\top}$
\State Generate $A_1 = Q \Gamma Q^{\top}$ such that $\Gamma=\text{diag}(\gamma)$ where $\gamma_B=0$, $\gamma_T>0$, and $\gamma_T\in\Rmbb^q$
\State Generate $\hat{y} \in \mathbb{R}^{m}$ and $\hat{A}_i \in \mathcal{S}^n$ for $i \in \{2,\dots, m\}$ 

\State Calculate $\hat{b}_i=\trace{(\hat{A}_iX^*)}$ for $i \in [m]$ and $\hat{C}=\sum_{i=1}^{m}\hat{y}_i\hat{A}_i+\hat{S}$

\State Build $Q_{(n+1)\times(n+1)}=\begin{pmatrix}
\hat{Q}& 0\\
0 & 1
\end{pmatrix}$
\State Generate positive diagonal matrix $\Sigma_B^0$, $\Sigma_T^0$, $\Sigma_N^0$, and number $\sigma_{n+1}^0>0$  
\State Build $X^*=Q\begin{pmatrix}
\Sigma_B&0& 0&0\\
0&0&0&0\\
0 &0& 0&0\\
0&0&0&0
\end{pmatrix}Q^{\top}$ and $X^0=Q\begin{pmatrix}
\Sigma_B^0&0& 0&0\\
0&\Sigma_T^0&0&0\\
0 &0& \Sigma_N^0 &0\\
0&0&0&\sigma_{n+1}^0
\end{pmatrix}Q^{\top}$
\State Generate positive diagonal matrix $\Lambda_B^0$, $\Lambda_T^0$, and $\Lambda_N^0$ 
\State Calculate $\delta=\sum_{i\in B}(\sigma_i-\sigma_i^0)\lambda_i^0+\sum_{i\in T}\sigma_i^0\lambda_i^0+\sum_{i\in N}\sigma_i^0(\lambda_i^0-\lambda_i)$
\State Generate $\lambda^0_{n+1}>(\frac{-\delta}{\sigma_{n+1}^0})^+$, and calculate $\lambda_{n+1}=\frac{\delta}{\sigma_{n+1}^0}+\lambda^0_{n+1}$ 
\State Build $S^*=Q\begin{pmatrix}
0&0& 0&0\\
0&0&0&0\\
0&0& \Lambda_N&0\\
0&0&0&\lambda_{n+1}
\end{pmatrix}Q^{\top}$ and $S^0=Q\begin{pmatrix}
\Lambda_B^0&0& 0&0\\
0&\Lambda_T^0&0&0\\
0&0& \Lambda_N^0&0\\
0&0&0&\lambda_{n+1}^0
\end{pmatrix}Q^{\top}$
\State Generate $y^0\in \mathbb{R}^{m+1}$ such that $y^0_{m+1}\not=0$ 
and build $y^*_{(m+1)}=\begin{pmatrix}
\hat{y}\\
0 
\end{pmatrix}$ 
\State Calculate $\alpha_i=\frac{1}{\sigma_{n+1}^0}\trace{\left(\hat{A}_i\hat{Q}\begin{pmatrix}
\Sigma_B-\Sigma_B^0&0& 0\\
0&-\Sigma_T^0&0\\
0 &0& -\Sigma_N^0 
\end{pmatrix}\hat{Q}^{\top}\right)}$ for $i \in [m]$
\State Build $A_i= \begin{pmatrix}
\hat{A}_i& 0\\
0 & \alpha_i
\end{pmatrix}$  for $i \in [m]$
\State $A_{m+1}= \frac{1}{y^0_{m+1}}\left(\sum_{i=1}^{m}(\yhat_i-y^0_i)\Ahat_i+Q\begin{pmatrix}
-\Lambda_B^0&0& 0&0\\
0&-\Lambda_T^0&0&0\\
0&0&\Lambda_T -\Lambda_T^0&0\\
0&0&0&\Lambda_{n+1} -\Lambda_{n+1}^0
\end{pmatrix}Q^{\top}\right)$
\State Calculate $C=\begin{pmatrix}
\hat{C}& 0\\
0 & \sum_{i=1}^{m}\hat{y}_i\alpha_i+\lambda_{m+1}
\end{pmatrix}$

\State Calculate  $b_i=\hat{b}_i$ for $i\in [m]$ and $b_{m+1}=\trace{(A_{m+1}X^*)}$
\State \textbf{Return }{SDOP $(A_1,\dots,A_m, b, C)$ with optimal solution $(X^*,y^*,S^*)$ and interior solution $(X^0,y^0,S^0)$}  

\end{algorithmic}
\end{algorithm}
\begin{theorem}\label{theo: maxcomSDOeigoptint}
For Algorithm~\ref{alg: SDO-IP-OP-Q-M},  solution $(X^*,y^*,S^*)$ is the maximally complementary optimal solution of the generated problem $(A_1,\dots,A_m, b, C)$.
\end{theorem}

\begin{proof}
The proof closely follows the proof of Theorem \ref{theo: maxcomSDOeigopt}; the only difference being that we expanded the matrices $A_i$ by adding a row and column. For constructing matrix $A_1$, the added eigenvalue $\gamma_{n+1}=\alpha_1$ and $Q_{n+1}=(0,0,0,\dots,0,1)^{\top}$ belong to partition $N$ where we do not have any restriction. Thus, adapting the proof of Theorem \ref{theo: maxcomSDOeigopt} to this theorem is straightforward.
\end{proof}
Among all proposed SDOP generators, Algorithm~\ref{alg: SDO-IP-OP-Q-M} provides the most sophisticated SDOPs, with maximally complementary and interior solutions in a general manner and gives opportunities for altering characteristics of an optimal solution, optimal partition, matrices $A_i$, $C$, and vector $b$ to study the performance of solution methods in a detailed and sophisticated analysis. However, this algorithm requires much more complicated computation than the other proposed generators. 
\section{Second-Order Cone Optimization}\label{sec: SOCO}
Before concluding, we adapt our techniques for LO and SDO to linear optimization problems over second order (or \textit{Lorentz}) cones. 

\subsection{Second Order Cone Optimization Problems}\label{sec: SOCOP}
A second-order cone is defined as follows
$$\left\{ (x_1, x_2, \dots, x_n) \in \R{n} : x_1^2 - \sum_{i=2}^n x_i^2 \geq 0, x_1 \geq 0 \right\}.$$
Observe that the above definition implies that $(x_1, x_2, \dots, x_n)$ is a second-order cone if and only if the matrix 
$$ \begin{pmatrix} x_1 & x^{\top}_{2:n} \\
x_{2:n} & x_1 I_{n-1} \end{pmatrix}$$
is positive semidefinite, where $x_{2:n}^{\top} \equiv (x_2, x_3, \dots, x_n)$ and $I_{n-1}$ is the identity matrix of order $n-1$. Accordingly, a primal or dual second-order cone optimization problem (SOCOP) may be interpreted as a special case of SDO \cite{sampourmahani2023semidefinite}.

In SOCO problems, we seek to minimize a linear objective function over a feasible region which is defined by the intersection of an affine space and the Cartesian product of $p$ second-order cones of dimension $n_i$, which is defined as 
$$\Lmbb^{n}=\Lcal^{n_1}\times\dots\times \Lcal^{n_p},\quad n=\sum_{i=1}^{p}n_i,$$
where
$$\Lcal^{n_i}=\{x^i=(x_1^i,\dots,x_{n_i}^i)^{\top}\in\mathbb{R}^{n_i}:x_1^i\geq\|x^i_{2:n_i}\|\},\quad i\in [p].$$
It is clear that LOPs are a special case of SOCOPs, where $n_i=1$ for $i \in [p]$. 

The primal and dual SOCO problems in standard form are represented as
\begin{align*}\label{e:SDO}
     z^P_{SOCO} &= \inf_x \{c^{\top}x :Ax = b~,~x\in\Lmbb^{n}\} ,\\
    z^D_{SOCO} &=  \sup_{(y, s)} \{ b^{\top} y:A^{\top}y + s = c~,~s\in\Lmbb^{n}  \},
\end{align*}
where $b\in\mathbb{R}^m$, $A=(A_1,\dots,A_p)$, $x=(x^1;\dots;x^p)$, $s=(s^1;\dots;s^p)$, and $c=(c^1;\dots;c^p)$, in which $A_i\in \mathbb{R}^{m\times n_i}$, $s^i\in\mathbb{R}^{n_i}$, and $c^i\in \mathbb{R}^{n_i}$ for $i\in [p]$. The set of primal and dual feasible solutions is defined as
$$\Pcal \Dcal_{SOCO}=\{(x,y,s)\in \Lmbb^{n}\times\mathbb{R}^{m}\times \Lmbb^{n}:Ax=b,A^{\top}y+s=c\}.$$
Let $$\Lcal_{+}^{n_i}=\{x^i\in\Lcal^{n_i}:x_1^i>\|x^i_{2:n_i}\|\},\quad i\in [p],$$
then we can define the set of primal and dual interior feasible solutions as
$$\Pcal \Dcal^0_{SOCO}=\{(x,y,s)\in \Lmbb^{n}_{+}\times\mathbb{R}^{m}\times \Lmbb^{n}_{+}:Ax=b,A^{\top}y+s=c\}.$$

Just as in LO and SDO, it is standard practice to assume the existence of an interior feasible primal-dual solution. With the existence of a strictly feasible solution, it follows that the Interior Point Condition (IPC) is satisfied \citep{mohammad2019conic}, guaranteeing that $z^P_{SOCO}=z^D_{SOCO}$ and the primal-dual optimal set
\begin{align*}
    \Pcal\Dcal_{SOCO}^* &= \left\{(x,y,s) \in \Pcal\Dcal_{SOCO}  ~:~  c^{\top}x = z^P_{SOCO} =  b^{T} y = z^D_{SOCO} \right\},
\end{align*}
is nonempty and bounded. Therefore, there exists an optimal primal-dual pair with zero duality gap. That is, for optimal solutions
$x^*$ and $(y^*, s^*)$, we have
\begin{equation}
   x^*\circ s^*=(x^1\circ s^1 ,\dots , x^p\circ s^p)  = 0,
\end{equation}
where the Jordan product ``$\circ$" is defined as
\begin{equation}
   x^i\circ s^i = \begin{pmatrix}(x^i)^{\top}s^i\\
   x_1^is_{2:n_i}^i+s_1^ix_{2:n_i}^i
   \end{pmatrix}.
\end{equation}
An optimal solution $(x^*, y^*, s^*)$ is called maximally complementary if $x^* \in \ri({\Pcal}_{SOCO}^*)$ and $(y^*; s^*) \in \ri({\Dcal}_{SOCO}^*)$. Further, $(x^*, y^*, s^*)$ is called strictly complementary if $$x^* + s^* \in \Lmbb^{n}_{+}.$$
\subsection{Instance Generators for SOCOPs}\label{sec: SOCOPGen}
Motivated by our work on LOP and SDOP  generators, we are further interested in applying these ideas to generate SOCO problems. Since SOCO can be interpreted as a special case of SDO, Sampourmahani et al.~\cite{sampourmahani2023semidefinite} studied mappings between SDOPs and SOCOPs and their optimal partitions. It is straightforward to develop SOCOP generators using the proposed SDOP generators augmented with the appropriate mapping. However, that route is not efficient, and we alternatively propose several SOCOP generators without using their SDO representation.

\subsubsection{SOCOPs with a Predefined Interior Solution}\label{sec: SOCOPINT}
Generating SOCOPs with an interior solution also ensures that the problem has an optimal solution with zero duality gap. Algorithm~\ref{alg: SOCO-IP} is a modification of Algorithms~\ref{alg: LO-IP} and~\ref{alg: SDO-IP} to generate SOCOPs with specific interior solutions.
\begin{algorithm}[H] 
\caption{Generating SOCO problems with a specific interior solution} \label{alg: SOCO-IP}
\begin{algorithmic}[1]
\State Choose dimensions $m<n$
\State Choose $n_1,\dots, n_p$ such that $n=n_1+\dots+n_p$
\State Generate $(x^0,s^0)$ such that $x^0\in\Lmbb_{+}^{n}$ and $s^0\in\Lmbb_{+}^{n}$ 
\State Generate $A\in\mathbb{R}^{m\times n}$ 
\State Generate $y^0 \in \mathbb{R}^{m}$ 
\State Calculate $b=Ax^0$ and $c=A^{\top} y^0+s^0$
\State \textbf{Return }{SOCOP $(A, b, c)$ with interior solution
$(x^0,y^0,s^0)$}
\end{algorithmic}
\end{algorithm}
To have an interior solution $x^0$, we must generate $(x^0)^i$ for $i=1,\dots, p$, such that $((x^0)^i_2,\dots,(x^0)^i_{n_i})\in\mathbb{R}^{n_i-1}$ and $(x^0)^i_1>\|(x^0)^i_{2:n_i}\|$. One way to generate such a solution is to generate $(x^0)^i\in\mathbb{R}^{n_i}$, and update it using the rule $$(x^0)^i_1=\|(x^0)^i_{2:n_i}\|+\|(x^0)^i_1\|.$$ 
 Similar to LO, if the matrix $A$ is generated randomly, then the probability of that all rows of $A$ are linearly independent is one. In addition, the user can generate a desired matrix $A$ with specific characteristics such as sparsity, condition number, and norm.
\subsubsection{SOCOPs with a Predefined Optimal Solution}\label{sec: SOCOPOPT}
For SOCOPs, the optimal partition is a bit more complicated than for LO and SDO. The index set $[p]$ is  partitioned to sets $(\Bcal, \Ncal,\Rcal, \Tcal_1,\Tcal_2,\Tcal_3)$ defined as
\begin{align*}
    \Bcal &\coloneqq \{  i  :  x_1^i > \|x_{2:n_i}^i\|_2, \text{ for some } x \in \Pcal_{SOCO}^* \},\\
    \Ncal &\coloneqq \{  i  :  x_1^i > \|x_{2:n_i}^i\|_2, \text{ for some } x \in \Dcal_{SOCO}^* \},\\
    \Rcal &\coloneqq \{  i  :  x_1^i = \|x_{2:n_i}^i\|_2 > 0, x_1^i = \|x_{2:n_i}^i\|_2 > 0,  \text{ for some } (x, y, x) \in \Pcal_{SOCO}^* \times \Dcal_{SOCO}^* \},\\
    \Tcal_1 &\coloneqq \{ i  :  x^i = x^i = 0, \text{ for all } (x, y, x) \in \Pcal_{SOCO}^* \times \Dcal_{SOCO}^*\},\\
    \Tcal_2 &\coloneqq \{ i  :  x^i = 0, \text{ for all } (y, x) \in \Dcal_{SOCO}^*, x_1^i = \|x_{2:n_i}^i\|_2 > 0, \text{ for some } x \in \Pcal_{SOCO}^*\},\\
    \Tcal_3 &\coloneqq \{ i  :  x^i = 0, \text{ for all } x \in \Pcal_{SOCO}^*, x_1^i = \|x_{2:n_i}^i\|_2 > 0, \text{ for some } (y, x) \in \Dcal_{SOCO}^*\}.
\end{align*}
For further discussion regarding the optimal partition in SOCOPs, we refer the reader to \citep{terlaky2014identification}. From here, we can develop Algorithm~\ref{alg: SOCO-OP} which is a generalization of Algorithm~\ref{alg: LO-OP} for generating random SOCOPs with specific optimal solutions.
\begin{algorithm}[H] 
\caption{Generating SOCO problems with a specific optimal solution} \label{alg: SOCO-OP}
\begin{algorithmic}[1]
\State Choose dimensions $m<n$
\State Choose $n_1,\dots, n_p$ such that $n=n_1+\dots+n_p$
\State Partition the index set $[p]$ to $(B, N, R, T_1, T_2, T_3)$  
\State For $i\in B$, $(s^*)^i=0$ and generate $(x^*)^i\in \mathbb{R}^{n_i}$  such $(x^*)^i_{1}>\|(x^*)^i_{2:n_i}\|$  
\State For $i\in N$, $(x^*)^i=0$ and generate $(s^*)^i\in \mathbb{R}^{n_i}$ such $(s^*)^i_{1}>\|(s^*)^i_{2:n_i}\|$  
\State For $i\in T_1$, $(s^*)^i=0$ and $(x^*)^i=0$
\State For $i\in T_2$, $(s^*)^i=0$ and generate $(x^*)^i\in \mathbb{R}^{n_i}$ such $(x^*)^i_{1}=\|(x^*)^i_{2:n_i}\|>0$  
\State For $i\in T_3$, $(x^*)^i=0$ and generate $(s^*)^i\in \mathbb{R}^{n_i}$ such $(s^*)^i_{1}=\|(s^*)^i_{2:n_i}\|>0$  
\State For $i\in R$, generate $(x^*)^i_{2:n_i}\in\mathbb{R}^{n_i-1}$ and $\delta\in\mathbb{R}$ and build 
$$(x^*)^i=\begin{pmatrix}\|(x^*)^i_{2:n_i}\|\\ (x^*)^i_{2:n_i}\end{pmatrix}\text{, and }(s^*)^i=\delta\begin{pmatrix}\|(x^*)^i_{2:n_i}\|\\ -(x^*)^i_{2:n_i}\end{pmatrix}$$
\State Generate $y^*\in \mathbb{R}^m$ 
\State Generate $A\in \mathbb{R}^{m\times n}$ 
\State Calculate $b=A x^*$ and $c=A^{\top} y^*+s^*$
\State \textbf{Return }{SOCOP $(A, b, c)$ with optimal solution
$(x^*,y^*,s^*)$}

\end{algorithmic}
\end{algorithm}
\begin{remark}
Algorithm \ref{alg: SOCO-OP} provides a SOCOP with an optimal solution, and that optimal solution may not be maximally complementary. Thus, the optimal partition $(\Bcal,\Ncal,\Rcal,\Tcal_1,\Tcal_2,\Tcal_3)$ of the generated problem may differ from $(B,N,R,T_1,T_2,T_3)$, and we only have
\begin{equation*}
    B\subseteq\Bcal,\ N\subseteq\Ncal,\ R\subseteq\Rcal,\  \Tcal_1\subseteq T_1,\  \Tcal_2\subseteq T_2 \cap T_1, \text{ and } \Tcal_3\subseteq T_3 \cap T_1.
\end{equation*}
\end{remark}
Similar to LOPs and SDOPs generators, one can generate $A$ random in a way to have specific characteristics. The norm and properties of $(x^*,y^*,s^*)$ are controllable directly. Also, the norm of $(b,c)$ can be predetermined by scaling $(x^*,y^*,s^*)$ appropriately and carefully, since determining the norm of all parameters and the optimal solution simultaneously is possible if the equations in line 11 of Algorithm~\ref{alg: SOCO-OP} hold. It is easy to see that $(x^*,y^*,s^*)$ is optimal since $x^*\circ s^*=0$ and it is feasible by construction. In the next section, we discuss how to generate a SOCOP with a maximally complementary solution.
\subsubsection{SOCOPs with a Predefined Maximally Complementary  Solution}\label{sec: SOCOPMC}
Since the optimal partition can affect the performance of algorithms to solve SOCOPs similar to SDO, we are interested in generating problems with predetermined optimal partitions. To do so, we adapt our instance generator for SDOPs with maximally complementary solution to SOCO in Algorithm~\ref{alg: SOCO-MC}. Let $A^p_{i,j}$ be the element in row $i$ and column $j$ of part (columns) of A that correspond to cone $p$. We also use the superscript to show the partition, e.g., $A^B$ denotes the columns of A which correspond to partition $B$.
\begin{algorithm}[H] 
\caption{Generating SOCOPs with a specific maximally complementary solution} \label{alg: SOCO-MC}
\begin{algorithmic}[1]
\State Choose dimensions $m<n$
\State Choose $n_1,\dots, n_p$ such that $n=n_1+\dots+n_p$
\State Partition the index set $[p]$ to $(B, N, R, T_1, T_2, T_3)$  such that $$|T_2|+1<m\leq|B|+|R|+|T_2|$$
\State For $i\in B$, $(s^*)^i=0$ and generate $(x^*)^i\in \mathbb{R}^{n_i}$  such $(x^*)^i_{1}>\|(x^*)^i_{2:n_i}\|$  
\State For $i\in N$, $(x^*)^i=0$ and generate $(s^*)^i\in \mathbb{R}^{n_i}$ such $(s^*)^i_{1}>\|(s^*)^i_{2:n_i}\|$  
\State For $i\in T_1$, $(s^*)^i=0$ and $(x^*)^i=0$
\State For $i\in T_2$, $(s^*)^i=0$ and generate $(x^*)^i\in \mathbb{R}^{n_i}$ such $(x^*)^i_{1}=\|(x^*)^i_{2:n_i}\|>0$  
\State For $i\in T_3$, $(x^*)^i=0$ and generate $(s^*)^i\in \mathbb{R}^{n_i}$ such $(s^*)^i_{1}=\|(s^*)^i_{2:n_i}\|>0$  
\State For $i\in R$, generate $(x^*)^i_{2:n_i}\in\mathbb{R}^{n_i-1}$ and $\delta\in\mathbb{R}$ and build 
$$(x^*)^i=\begin{pmatrix}\|(x^*)^i_{2:n_i}\|\\ (x^*)^i_{2:n_i}\end{pmatrix}\text{, and }(s^*)^i=\delta\begin{pmatrix}\|(x^*)^i_{2:n_i}\|\\ -(x^*)^i_{2:n_i}\end{pmatrix}$$
\State Generate $y^*\in \mathbb{R}^m$ 
\State Generate $A\in \mathbb{R}^{m\times n}$ such that 
\begin{itemize}
    \item First row:
    \begin{align*}
    A_{1,1}^p>0,  A_{1,j}^p=0&\text{ for } j=2,\dots,n_p, \text{ and } p\in T_1\cup T_3\\
      A_{1,j}^p=0&\text{ for } j=1,\dots,n_p, \text{ and } p\in B\cup R \cup T_2\\
      A_{1,j}^p\in \Rmbb &\text{ for } j=1,\dots,n_p, \text{ and } p\in N
\end{align*}
\item Row $2$ to $|T_2|+1$:
\begin{align*}
      A^p_{p,1:n_p}=\begin{bmatrix}
      -1&\frac{(x^{*k}_{2:n_p})^{\top}}{\|x^{*k}_{2:n_p}\|}\end{bmatrix}, A^p_{k,1:n_k}=0&\text{ for all } k\not = q, \text{ and }p\in T_2
\end{align*}
\item The other rows should be generated such that rank$([A^B,A^R,A^{T_2}])=m$.
\end{itemize}

\State Calculate $b=A x^*$ and $c=A^{\top} y^*+s^*$
\State \textbf{Return }{SOCOP $(A, b, c)$ with maximally complementary solution
$(x^*,y^*,s^*)$}

\end{algorithmic}
\end{algorithm}
Compared to Algorithm~\ref{alg: SOCO-OP}, Algorithm~\ref{alg: SOCO-MC} imposes more restrictions on how $A$ is generated.
\begin{theorem}
For any SOCOP $(A,b,c)$ generated by Algorithm~\ref{alg: SOCO-MC}, the generated optimal solution $(x^*,y^*,s^*)$ is maximally complementary. 
\end{theorem}
\begin{proof}
One can verify that $b_1=0$, and the first row of $A$ enforces any optimal solution $\Bar{x}$ to satisfy
$$\Bar{x}^p=0 \text{ for all } p\in T_1\cup T_3.$$
From constraint $2$ to $|T_2|+1$, we add a constraint for each cone $p$ in partition $T_2$ in which coefficients are zero for all variables except for variables in cone $p$. Since the corresponding right-hand side is zero and the coefficients are the normal vector to the cone $p$ at the point $x^*$, all feasible solutions must lie on the ray which is on the boundary of the cone $p$ and passing through the point $x^*$. Thus, for any optimal solution $\Bar{x}$, we have 
$$\Bar{x}^p_1=\|x^{p}_{2:n_p}\| \text{ for all } p\in T_2.$$

Up to this point, we have shown that $x^*\in ri(\Pcal^*)$, and the last part of the proof is to establish that the dual problem has a unique optimal solution $(y^*,s^*)$. To prove it, let assume that it has another optimal solution $(\bar{y},\bar{s})$, and $X^*=\text{diag}(x^*)$. Then, we have
$$X^*A^{\top}(\ybar-y^*)=-X^*(\sbar-s^*)=0.$$
Since $A$ generated in a way that the rank of $X^*A^{\top}$ is $m$, we have $\ybar=y^*$. We can conclude that $(x^*, y^*,s^*)$ is a maximal complementary solution for the generated problem.
\end{proof}
\begin{corollary}
For any SOCOP generated by Algorithm~\ref{alg: SOCO-MC}, the optimal partition $(\Bcal,\Ncal,\Rcal,\Tcal_1,\Tcal_2,\Tcal_3)$ is equal to $(B,N,R,T_1,T_2,T_3)$.
\end{corollary}
As expected, generating SOCOPs with predefined optimal partitions restricts on how matrix $A$ is generated. However, some components of $A$ are not restricted, and enable the user to control the properties of $A$. This is discussed next.
\subsubsection{SOCOPs with Optimal and Interior Solutions}\label{sec: SOCOPINTOPT}
We can extend Algorithm~\ref{alg: SOCO-OP} to provide both specific interior and optimal solutions by adding one row and column to the matrix $A$. We aim to generate optimal and interior solutions in a general manner, but we need to enforce the orthogonality condition:
\begin{equation}\label{con3}
    (x^0-x^*)^{\top} (s^0-s^*)=0.
\end{equation}
Note that this is a natural requirement; similar to LOPs, we have $(x^*-x^0)\in \text{Lin}^{\perp}(A)$ and $(s^*-s^0)\in \text{Lin}(A)$. 
\begin{algorithm}[H] 

\caption{Generating SOCOPs with specific interior and optimal solutions} \label{alg: SOCO-IP-OP}
\begin{algorithmic}[1]
\State Choose dimensions $m<n$ (the dimension of generated SOCOP: $m+1,n+1$)
\State Generate $(\Ahat,\bhat,\chat)$ with optimal solution $(\xhat,\yhat,\shat)$ by Algorithm~\ref{alg: SOCO-OP} (dimension: $p\times m$)

\State Generate $x^0\in\Lmbb_{+}^{n_1,\dots,n_p+1}$ and let $\xhat^0=((x^0)^1;\dots;(x^0)^p_{1:n_p}) $
\State Build $x^*\in\Lmbb_{+}^{n_1,\dots,n_p+1}$ such that $$(x^*)^{1:p-1}=(\xhat)^{1:p-1},\ (x^*)^{p}_{1:n_p}=(\xhat)^{p}_{1:n_p}\text{, and }(x^*)^{p}_{n_p+1}=0$$  
\State Calculate vector $\alpha=\frac{1}{(x^0)^{p}_{n_p+1}}\Ahat(\xhat-\xhat^0)$
\State Build $\Atilde=(\Ahat,\alpha)$ (concatenation of a column to matrix) \label{sconcate}

\State Generate $y^0 \in \mathbb{R}^{m+1}$ such $y^0_{m+1}\not =0$, and build $y^*=(\yhat,0)^{\top} \in \mathbb{R}^{m+1}$
\State Generate $\shat^0\in\Lmbb_{+}^{n_1,\dots,n_p}$ 
\State Let $\delta=(\xhat^0-\xhat)^{\top} (\shat^0-\shat)$, and generate $(s^0)^p_{n_p+1}>(\frac{-\delta}{(x^0)^p_{n_p+1}})^+$ \label{a13s10}
\State Calculate $\hat{s}^p_{n_p+1}=\frac{\delta}{(x^0)^p_{n_p+1}}+(s^0)^p_{n_p+1}$
\State Build $s^*=(\shat,\hat{s}^p_{n_p+1})^{\top}$, $s^0=(\shat^0,(s^0)^p_{n_p+1})^{\top}$; \label{a13s9}
\State Calculate vector $\beta=\frac{1}{y^0_{m+1}}(\Atilde^{\top} (\yhat-(y^0)_{1:m})+s^*-s^0)$
\State Build $A=\begin{pmatrix}\Atilde\\\beta^{\top} \end{pmatrix}$ (concatenation of a row to matrix)
\State Calculate $b=Ax^*$ and $c=A^{\top}y^*+s^*$
\State \textbf{Return }{SOCOP $(A, b, c)$ with interior solution
$(x^0,y^0,s^0)$ and optimal solution $(x^*,y^*,s^*)$}
\end{algorithmic}
\end{algorithm}
Theorem \ref{t3} shows that the claimed properties of  $(x^0,y^0,s^0)$ and  $(x^*,y^*,s^*)$ are indeed correct. Before presenting  Theorem~\ref{t3}, we need to verify the orthogonality properties of the generated solution.
\begin{lemma}\label{lem: orthogonality of SOCO}
For any $(x^0,y^0,s^0)$ and  $(x^*,y^*,s^*)$ generated by Algorithm \ref{alg: SOCO-IP-OP}, we have
$$(x^0-x^*)^{\top} (s^0-s^*)=0.$$
\end{lemma}
\begin{proof}
Similar to the proof of Lemma \ref{lem: orthogonality of LO}, Steps \ref{a13s10} and \ref{a13s9} of Algorithm~\ref{alg: SOCO-IP-OP}  ensure that the orthogonality condition holds.
\end{proof}
Using Lemma~\ref{lem: orthogonality of SOCO}, the following theorem shows that the generated problem satisfies the desired properties.
\begin{theorem}\label{t3}
Let $(x^0,y^0,s^0)$ and  $(x^*,y^*,s^*)$ be generated by Algorithm~\ref{alg: SOCO-IP-OP}. Then, 
\begin{subequations}
\begin{align*}
    x^* \circ s^*&=0,\\
    Ax^*&=b,\\
    A^{\top} y^*+s^*&=c,\\
    Ax^0&=b,\\
    A^{\top} y^0+s^0&=c,\\
x^*\in \Lmbb ^{n+1},~s^*\in \Lmbb^{n+1},~x^0\in \Lmbb^{n+1}_+,~s^0&\in \Lmbb^{n+1}_+.
\end{align*}
\end{subequations}
That is, $(x^0,y^0,s^0)$ and  $(x^*,y^*,s^*)$ are interior and optimal solutions, respectively, for the generated SOCOP $(A,b,c)$.
\end{theorem}
\begin{proof}
The proof is similar to the proof of Theorem~\ref{t1}.
\end{proof}
Compared to the SDOP generators, the SOCOP generators are computationally simpler since they do not require generating random orthonormal or positive semidefinite matrices. Let $t_R$ be the amount of arithmetic operations required to generate a number randomly. To generate orthonormal or positive semidefinite matrices, we need to use a decomposition method, which requires $\Ocal(n^3)$ arithmetic operations as discussed in the appendix. In the general case, the LOP and SOCOP  generators require $\Ocal(n^2t_r)$ arithmetic operations, while the SDO generators require $\Ocal(n^3t_r)$ arithmetic operations. It should be mentioned that if we want to generate a random matrix $A$ in LOPs and SOCOPs with specific condition numbers, then we need to use decomposition methods and the complexity of the LOP and SOCOP generators increases to $\Ocal(n^3t_r)$ arithmetic operations.
\subsubsection{SOCOPs with Predefined Interior and Maximally Complementary Solutions}\label{sec: SOCOPINTMC}
To generate a SOCOPs with both interior and maximally complementary solutions, we can use Algorithm~\ref{alg: SOCO-IP-OP} and in its first step, use Algorithm~\ref{alg: SOCO-MC} which provide a SOCOP with maximally complementary solution. The only difference is that we should choose the partition such that the last cone is in the partition $N$. By this modification, the added column in Step~\ref{sconcate} of Algorithm~\ref{alg: SOCO-IP-OP}   will be in partition $N$, which satisfies all the restrictions needed to keep $(x^*,y^*,s^*)$ maximally complementary. In this way, we can generate a SOCOP with interior solution and predetermined optimal partition.
\section{Implementation}\label{sec: imp}
All mentioned generators are implemented in a \texttt{python} package, which is available in open source at
\url{https://github.com/qcol-lu/qipm}. 
This package gives the option of prescribing the norm of vectors, condition numbers, and sparsity of the matrices. In addition, several versions of interior point methods, such as feasible/infeasible, exact/inexact, and  long-step/short-step/predictor-corrector, are implemented and available for the experiment. There is also an option to choose the solver of the Newton system. One may choose classical or quantum linear system algorithms.

\section{Conclusion} \label{sec: con}
We develop and implement several random instance generators for LO, SDO, and SOCO with specific optimal and/or interior solutions. Because of high level of controllability, these generators enable users to analyze different features of the problem, such as sparsity and condition number, to study the performance of different algorithms smartly. In addition, we proposed SDOP and SOCOP generators with predefined optimal partition, which can be used to generate computationally challenging instances. 

The proposed generators can also be used to study the average performance of algorithms for solving LO, SDO, and SOCO problems with different probability distributions for input data, optimal and interior solutions. Future research directions include expanding the construction of these generators for other classes of conic, polynomial, and nonlinear optimization problems.

A useful direction for extending the proposed generators is to generate hard problems which are challenging for algorithms and solvers, e.g., LOPs which are primal or dual unbounded. For SDO and SOCO, it is worth exploring to develop generators which produce instances that have zero-duality gap, but with an optimal solution that is not attainable or instances with non-zero duality gap.
\section{Acknowledgement}
This work is supported by Defense Advanced Research Projects Agency as part of the project W911NF2010022: {\em The Quantum
Computing Revolution and Optimization: Challenges and Opportunities}.

\bibliographystyle{dependencies/tfs}
\bibliography{dependencies/interacttfssample}

\section{Appendices}
\appendix
Here, we review some basic procedures to generate random positive semidefinite matrices and orthogonal matrices, which can be used in the proposed SDOP and SOCOP generators. 
\section{Generating Random Positive Semidefinite Matrices}\label{appendix:PSD}
There are several approaches to generating a positive semidefinite matrix $P\in \mathcal{S}_+^n$.
\begin{enumerate}

    \item Generate a random matrix $A\in \mathbb{R}^{n\times n}$ and calculate the target matrix $P=AA^{\top}$. If $A$ has full rank with probability 1, the matrix $P$ is positive semidefinite with probability 1. 
    \item A more efficient way is to generate a lower triangular random matrix $L\in \mathbb{R}^{n\times n}$ and calculate the target matrix $P=LL^{\top}$. If the diagonal elements of $L$ are non-zero, then $P$ is positive definite. If some of the diagonal elements of $L$ are zero, then $P$ is positive semidefinite.
    \item Generate an orthonormal matrix $Q\in \mathbb{R}^{n\times n}$ and a positive diagonal matrix $\Lambda$. Then calculate $P=Q\Lambda Q^{\top}$. If the diagonal elements of $\Lambda$ are greater than zero, then $P$ is positive definite. If diagonal elements of $\Lambda$ are greater than or equal to zero, then $P$ is positive semidefinite.
    \item Since generating a random orthonormal matrix is computationally expensive, we can generate a lower triangular random matrix $L\in \mathbb{R}^{n\times n}$ in which all diagonal elements are one instead. Then we can compute the target matrix as $P=LD L^{\top}$ where $D$ is a diagonal matrix with non-negative elements.
\end{enumerate}
The second one requires the fewest arithmetic operations among the four mentioned approaches. However, the third one gives the option of determining the range of eigenvalues of the matrix $P$, which is helpful in predetermining the condition number of matrix $P$.
\section{Generating Random Orthogonal Matrices}\label{appendix:ortho}
A general approach is to generate a random matrix $A\in \mathbb{R}^{n\times n}$ and orthogonalize it by the Gram–Schmidt process or other methods in QR decomposition, such as Modified Gram–Schmidt and Householder methods. Generating random orthogonal (or unitary) matrices is an active research area and there are many efficient procedures to generate such matrices, e.g., see \citep{mezzadri2006generate}.

\end{document}